\documentclass[11pt,a4paper,reqno]{amsart}
\usepackage[applemac]{inputenc}
\usepackage[T1]{fontenc}
\usepackage{amsmath}
\usepackage{amsthm}
\usepackage{amsfonts}
\usepackage{amssymb}
\usepackage{graphicx}
\usepackage{color}
\usepackage{amsbsy}
\usepackage{mathrsfs}
\usepackage{bbm}
\newtheorem{theorem}{Theorem}[section]
\newtheorem{lemma}[theorem]{Lemma}

\newtheorem{proposition}[theorem]{Proposition}

\theoremstyle{remark}
\newtheorem{remark}[theorem]{\it \bf{Remark}\/}

\numberwithin{equation}{section}
\catcode`@=11
\def\section{\@startsection{section}{1}%
  \z@{1.5\linespacing\@plus\linespacing}{.5\linespacing}%
  {\normalfont\bfseries\large\centering}}
\catcode`@=12
\newcommand{\be}{\begin{equation}}
\newcommand{\ee}{\end{equation}}
\newcommand{\bea}{\begin{eqnarray}}
\newcommand{\eea}{\end{eqnarray}}
\newcommand{\bee}{\begin{eqnarray*}}
\newcommand{\eee}{\end{eqnarray*}}
\def\div{{\rm div \;}}

\def\pa{\partial}

\def\RR{\mathbb{R}}

\def\ds{\displaystyle}

\catcode`@=11
\def\supess{\mathop{\operator@font Sup\,ess}}
\catcode`@=12
\def\curl{{\rm curl \;}}

\def\div{{\rm div \;}}

\def\RR{\mathbb{R}}

\def\ds{\displaystyle}

\def\bar#1{{\overline #1}}

\def\R2+{\RR ^2_+}

\def\pa{\partial}

\def\lim{\mathop{\rm lim}}

\def\et{\tilde{\e}}
\def\te{\tilde{\e}}

\def\pa{\partial}

\def\et{\tilde{\e}}

\def\pa{\partial}

\def\matchal{\mathcal}

\begin{document}
\title[]{On small traveling waves to the mass critical fractional NLS}
\author[I. Naumkin]{Ivan Naumkin}
\address{Laboratoire J.A. Dieudonn\'e, Universit\'e de la C\^ote d'Azur, France}
\email{ivan.naumkin@unice.fr}\author[P. Rapha\"el]{Pierre Rapha\"el}
\address{Laboratoire J.A. Dieudonn\'e, Universit\'e de la C\^ote d'Azur, France}
\email{praphael@unice.fr}
\begin{abstract}
We consider the mass critical fractional (NLS)
$$
i\partial_{t}u-\left\vert D\right\vert ^{s}u+u\left\vert u\right\vert
^{2s}=0,\text{ \ }x\in\mathbb{R},\text{ \ }1<s<2.%
$$
We show the existence of travelling waves for all mass below the ground state mass, and give a complete description of the associated profiles in the small mass limit. We therefore recover a situation similar to the one discovered in \cite{GLPR} for the critical case $s=1$, but with a completely different asymptotic profile when the mass vanishes.
\end{abstract}

\maketitle
\newcommand{\ag }{\widetilde{\mathbf A}}
\newcommand{\av }{\overline{\mathbf A}}
\newcommand{\ap }{\mathbf A}
\newcommand{\ere}{ {\mathbb R}}
\newcommand{\ZETA}{{\mathbb Z}}
\newcommand{\ese}{{\mathbb S}}
\newcommand{\CE}{{\mathbb C}}
\newcommand{\ls}{L^2(\ese^{n-1})}
\newcommand {\erc}{\mathcal R}
\def\p2{\mathcal A_{\Phi,2\pi}(B)}
\def\0p2{\mathcal A_{\Phi,2\pi}(0)}
\def\sp2{\mathcal A_{\Phi,2\pi,\hbox{\rm SR}}(B)}
\def\beq{\begin{equation}}
\def\ene{\end{equation}}
\def \ds {\displaystyle}
\newcommand{\bull}{\hfill $\Box$}
\def\qed{\ifhmode\unskip\nobreak\fi\ifmmode\ifinner
\else\hskip5pt\fi\fi\hbox{\hskip5pt\vrule width4pt height6pt
depth1.5pt\hskip1pt}}
\def\v{\mathbf v}
\def\hu{\hat{\mathbf u}}
\def\hv{\hat{\mathbf v}}
\def\hn{\hat{\mathbf n}}
\def\hw{\hat{\mathbf w}}
\def\curl{\, \hbox{ \rm curl}\,}
\def\mo{\mathbf p}
\def\ta{\tilde{A}}
\def \tf{\tilde{\phi}}
\def \ts{\tilde{\psi}}
\def\tB{\tilde{B}}
\def\tvf{\tilde{\varphi}}
\def\et{e^{-iz H_1}}
\def\te{\tilde{\varepsilon}}
\def\td{\tilde{\delta}}
\def\div{\,\hbox{\rm div}\,}
\def\xin{x_{\rm in}}
\def\xout{ x_{\rm out}}
\def\+out{x^{\rm out}}

\section{Introduction.}

\subsection{Setting of the problem} We study the existence and uniqueness of traveling waves for the mass critical
fractional nonlinear Schr\"{o}dinger equation
\begin{equation}
\label{1tw}
i\partial_{t}u-\left\vert D\right\vert ^{s}u+u\left\vert u\right\vert
^{2s}=0,\text{ \ }x\in\mathbb{R},\text{ \ }1<s<2,
\end{equation}
where
\[
D:=-i\partial_{x},\text{ \ \ }\mathcal{F}\left(  \left\vert D\right\vert
u\right)  =\left\vert \xi\right\vert \left(  \mathcal{F}u\right)  \left(
\xi\right),
\]
which appear as limiting models in various physical situations, see  \cite{lenzman}, \cite{Frohlich} and references therein. The existence of the ground state solution $u(t,x)=Q_s(x)e^{it}$, $Q_s>0$, of \eqref{1tw} follows from classical variational arguments, and uniqueness is a deep result \cite{lenzman}. The ground state produces a sharp criterion of global existence: for all $u_0\in H^{\frac s2}(\mathbb R)$ with $\|u_0\|_{L^2}<\|Q_s\|_{L^2}$, $\exists !u\in \mathcal C^0([0,\infty),H^{\frac s2})$ solution to \eqref{1tw}, \cite{gerard}, and there exists a minimal blow up solution at the threshold $\|u_0\|_{L^2}=\|Q_s\|_{L^2}$, \cite{KLR}. For $\|u_0\|_{L^2}<\|Q\|_{L^2}$, the behaviour of solutions dramatically depends on $s$:\\

\noindent\underline{$s=2$}: In the local case, all solutions below the ground state scatter, which is an elementary consequence of the pseudo-conformal symmetry for $u_0\in H^1\cap \{xu \in L^2\}$ and follows from the Kenig-Merle route map \cite{KM} coupled to Morawetz like estimates, \cite{dodson}. Note that in this case, the travelling wave family generated by the ground state solitary wave is explicitly given by the action of Galilean symmetry $$u(t,x)=Q_{s,\beta}(x-2\beta t)e^{i|\beta|^2t}, \ \ Q_{s,\beta}(y)=e^{i\beta \cdot y} Q_s(y), \ \ \beta\in \mathbb R$$ and hence the explicit degeneracy
\be
\label{degenemass}
\forall \beta\in \Bbb R, \ \ \|Q_{s,\beta}\|_{L^2}=\|Q_s\|_{L^2}.
\ee
\noindent\underline{$s=1$}: The half wave case is treated in details in \cite{GLPR} where the existence of travelling $u(t,x)=Q_{s,\beta}(x-\beta t)$ is proved with $$\lim_{\beta\uparrow 1}\|Q_{s,\beta}\|_{L^2}=0.$$ In fact, a unique branch is constructed with the asymptotic behaviour
\be
\label{cnkenoeneo}
Q_{s,\beta}(x)=\left(Q_++o_{\beta \to 1}(1)\right)\left(\frac{x}{1-\beta}\right)\ \ \mbox{as}\ \ \beta\uparrow 1
\ee where $Q_+$ is the ground state to the limiting non-local Szeg\H{o} equation $$\pa_tu=\Pi_+(u|u|^2), \ \ \widehat{\Pi_+u}={\bf 1}_{\xi>0}\hat{u},$$ see \cite{gerard,poc}. The existence of the critical speed $\beta=1$ is the starting point for the construction of two bubbles interacting solitons with growing Sobolev norms, \cite{poc}, \cite{GLPR}.

\subsection{Statement of the result} Our aim in this paper is to investigate the case $1<s<2$ and show that a third scenario occurs. Let us consider the travelling wave problem. We define
\[
u_{\beta}\left(  t,x\right)  :=e^{it\gamma}Q_{\beta}\left(  x-2\beta t\right)
,
\]
where $\gamma=\gamma\left(  \beta\right)  \in\mathbb{R}$. Then, in order to
solve (\ref{1tw}), $Q_{\beta}$ must satisfy the equation%
\begin{equation}
\left[  \left\vert D\right\vert ^{s}-2\beta D+\gamma\right]  Q_{\beta
}=Q_{\beta}\left\vert Q_{\beta}\right\vert ^{2s}\label{tw3}%
\end{equation}
which is the Euler-Lagrange equation corresponding to the minimization problem%
\begin{equation}
E_{\beta}\left(  N\right)  :=\inf\left\{  \mathcal{E}_{\beta}\left(  u\right)
:u\in H^{s/2}\left(  \mathbb{R}\right)  \text{ and }\int_{\mathbb{R}%
}\left\vert u\right\vert ^{2}=N\right\}  \label{tw4}%
\end{equation}
where%
\begin{equation}
\mathcal{E}_{\beta}\left(  u\right)  :=\frac{1}{2}\int_{\mathbb{R}}\bar
{u}\left\vert D\right\vert ^{s}u-\beta\int_{\mathbb{R}}\bar{u}Du-\frac
{1}{2s+2}\int_{\mathbb{R}}\left\vert u\right\vert ^{2s+2}.\label{tw107}%
\end{equation}
For $\beta=0$, there holds the Gagliardo--Nirenberg-type inequality (see \cite{Weinstein} and
\cite{Weinstein1})
\begin{equation}
\int_{\mathbb{R}}\left\vert u\right\vert ^{2s+2}\leq C_{s}\left(
\int_{\mathbb{R}}\bar{u}\left\vert D\right\vert ^{s}u\right)  \left(
\int_{\mathbb{R}}\left\vert u\right\vert ^{2}\right)  ^{s},\label{tw5}%
\end{equation}
where $C_{s}=\frac{s+1}{\left\langle Q,Q\right\rangle ^{s}},$ and the ground state $Q_s$
is the optimizer of (\ref{tw5}). We may now state the main results of this paper where to ease notations, we note $Q_s=Q$.\\

\noindent\underline{Existence of a minimizer} for all mass below the ground state.
\begin{theorem}[Existence of a minimizer]
\label{Th1}Let $1<s<2$ and $\beta\geq0$. Then, for all $0<N<\|Q\|_{L^2} ,$ the problem (\ref{tw4}) has a minimizer $Q_{\beta,N}\in
H^{s/2}\left(  \mathbb{R}\right)  $ with $\|Q_{\beta,N}\|_{L^2}^2 =N$ that satisfies (\ref{tw3})  for some $\gamma
=\gamma\left(  \beta,N\right)  \in\mathbb{R}$.
\end{theorem}

\noindent\underline{Asymptotic as $N\to 0$}. For $0<N<\|Q\|^2_{L^2} $ and $\beta\geq0$, we denote by
$\mathbf{Q}_{\beta,N}$ the set of minimizers of the problem (\ref{tw4})$,$
which is not empty by the previous result. Let $\mathcal{R}\in H^{1}$ be the
unique positive, radial symmetric solution of
\begin{equation}
-\triangle\mathcal{R+}\lambda\left(  s\right)  \mathcal{R-R}^{2s+1}=0
\label{tw75}%
\end{equation}
with $\lambda\left(  s\right)  :=\left(  \frac{s\left(  s-1\right)  }{2}%
\rho_{0}^{s}\right)  ^{-\frac{2}{2-s}}$, $\rho_{0}=\int\left\vert
\mathcal{R}_{0}\right\vert ^{2}$ and $\mathcal{R}_{0}$-the solution to the equation
$-\triangle\mathcal{R}_{0}\mathcal{+R}_{0}\mathcal{-R}_{0}^{2s+1}=0$ (see
Proposition \ref{P1} below). For $\beta\geq0,$ we denote
\[
\xi^{\ast}=\left(  \frac{2\beta}{s}\right)  ^{\frac{1}{s-1}}.
\]

\begin{theorem}[Asymptotics of $Q_{\beta,N}$ as $N\to0$]
\label{Th2}Let $0<N<\|Q\|^2_{L^2},$ $\beta\geq0$ and
$Q_{\beta,N}\in\mathbf{Q}_{\beta,N}.$ Consider the function%
\begin{equation}
R_{N}\left(  x\right)  =\left(  \frac{s\left(  s-1\right)  }{2}\right)
^{-\frac{1}{2s}}e^{-i\left(  \xi^{\ast}N^{-\frac{s}{2-s}}x\right)  }%
N^{-\frac{1}{2-s}}\left(  \xi^{\ast}\right)  ^{-1/2}Q_{\beta,N}\left(
\frac{x}{N^{\frac{s}{2-s}}\xi^{\ast}}\right)  . \label{tw117}%
\end{equation}
Then, there exist $\tilde{x},\tilde{\gamma}\in\mathbb{R}$, $\tilde{\gamma
}=\tilde{\gamma}\left(  N\right)  $ and $\tilde{x}=\tilde{x}\left(  N\right)
,$ such that
\begin{equation}
\lim_{N\rightarrow0}\left\Vert e^{i\tilde{\gamma}}R_{N}\left(  \cdot+\tilde
{x}\right)  -\mathcal{R}\right\Vert _{H^{r}}=0, \label{tw171}%
\end{equation}
for any $r\geq0,$ uniformly with respect to $\beta\geq0$. Moreover, suppose
that $Q_{\beta,N}$ solves (\ref{tw3}) with the Lagrange multiplier
$\gamma\left(  \beta,N\right)  \in\mathbb{R}$ and let $\theta_{N}$ be given by $\theta_{N}=\frac{2}%
{s}\left(  \left(  s-1\right)  ^{-1}\left(  \frac{2\beta}{s}\right)
^{-\frac{s}{s-1}}\gamma\left(  \beta,N\right)  -1\right)  .$ Then,%
\[
\left\vert \theta_{N}-\lambda\left(  s\right)  \right\vert =o\left(  1\right)
,\text{ as }N\rightarrow0,
\]
uniformly with respect to $\beta\geq0$.
\end{theorem}

\noindent\underline{Uniqueness for small mass}.

\begin{theorem}[Uniqueness for small mass]
\label{Th3}There exists $0<N_{0}<\|Q\|_{L^2}^2 ,$ such that
the following holds. Given $0<N<N_{0},$ $\beta\geq0$ and $Q_{\beta,N}%
,\tilde{Q}_{\beta,N}\in\mathbf{Q}_{\beta,N},$ there exist $\phi,y\in
\mathbb{R}$ such that
\[
\tilde{Q}_{\beta,N}\left(  x\right)  =e^{i\phi}Q_{\beta,N}\left(  x-y\right)
.
\]
\end{theorem}

\noindent\underline{Control of the tails}. Finally, we have a complete description of the tail of solutions for small mass. Let \begin{equation}
\mathcal{C}_{1}=\sqrt{\frac{\pi}{2\lambda\left(  s\right)  }}\text{ \ and
}\mathcal{C}_{2}=\dfrac{\left(  si^{s+1}+\left(  -i\right)  ^{s+1}\right)
e^{-i\frac{x}{\kappa}}}{2\sqrt{2\pi}\left(  s-1\right)  }\Gamma\left(
s\right)  , \label{tw168}%
\end{equation}
where $\Gamma\left(  s\right)  $ denotes the Gamma function.

\begin{theorem}[Tail asymptotics for small mass]
\label{Th4}Let $R_{N}$ be defined by (\ref{tw117}) and $\tilde{x}%
,\tilde{\gamma}\in\mathbb{R}$ be such that (\ref{tw171}) holds. Then, the
following asymptotics are valid%
\begin{align*}
e^{i\tilde{\gamma}}R_{N}\left(  x+\tilde{x}\right)   &  =\mathcal{C}%
_{1}e^{-\sqrt{\lambda\left(  s\right)  }\left\vert x\right\vert }\int
e^{\sqrt{\lambda\left(  s\right)  }y}\left(  \left\vert \mathcal{R}\right\vert
^{2s}\mathcal{R}\right)  \left(  y\right)  dy+\frac{\mathcal{C}_{2}%
N^{\frac{s\left(  2+s\right)  }{2-s}}}{\left\vert x\right\vert ^{s+1}}%
\int\left\vert \mathcal{R}\right\vert ^{2s}\mathcal{R}\\
&  +\left(  e^{-\sqrt{\lambda\left(  s\right)  }\left\vert x\right\vert
}+\frac{N^{\frac{s\left(  2+s\right)  }{2-s}}}{\left\vert x\right\vert ^{s+1}%
}\right)  \left(  o_{\left\vert x\right\vert }\left(  1\right)  +o_{N}\left(
1\right)  \right)  ,
\end{align*}
where $o_{N}\left(  1\right)  \rightarrow0,$ as $N\rightarrow0,$ and
$o_{\left\vert x\right\vert }\left(  1\right)  \rightarrow0,$ as $\left\vert
x\right\vert \rightarrow\infty.$
\end{theorem}

\noindent{\it Comments on the results}.\\

{\it 1. Existence and uniqueness}. The existence proof follows the path \cite{Frohlich} which adapts the classical concentration compactness argument \cite{Lions}. Let us say that we focused on dimension $d=1$ only for the sake of simplicity, but clearly the argument can be extended to higher dimensions as well. Uniqueness in the small mass limit requires a careful renormalization on the Fourier side and the sharp understanding of the role Galilean drifts which generate an explicit symmetry group for $s=2$ only. Related renormalization occur for example in \cite{mizu} for the description of high momentum solitary waves. Note that like the case $s=1$ \eqref{cnkenoeneo}, a concentration phenomenon occurs in the limit $N=0$, but the associated profile corresponds to a {\it local} limiting (NLS) problem, profile $\matchal R$, and concentration occurs with large Galilean like oscillations, \eqref{tw117}.\\

{\it 2. Tails and interaction}. The computation of the tail of the travelling wave in Theorem \ref{Th4} relies on a careful computation of the Fourier side. Related results for the travelling waves of the Gross Pitaevski equation are given in \cite{gravejat}. In \cite{GLPR}, the sharp description of the tail of the travelling wave is an essential step for the derivation of the modulation equations associated to energy exchanges between two interacting solitary waves. The derivation of related modulation equations for $1<s<2$ and the description of multiple bubbles interaction is a challenging problem due to the presence of additional high Galilean like oscillations, but Theorems \ref{Th3}-\ref{Th4} are the necessary starting point for such an investigation.\\

The paper is organized as follows. In Section 2 we prove Theorem \ref{Th1}. We
translate problem (\ref{tw4}) to a $\beta-$independent problem and we prove
the existence of minimizers for this translated problem. Section 3 is devoted
to the proof of Theorem \ref{Th2}. Again, we translate (\ref{tw4}) to a
problem where the mass of the minimizers is independent on $N$. We obtain an
asymptotic expansion for small $N$ for the minimizers of this new problem and
for the corresponding Lagrange multipliers. These yield the results of Theorem
\ref{Th2}. Section 4 is dedicated to the proof of Theorem \ref{Th3}. Finally,
in Section 5 we prove Theorem \ref{Th4}.

\subsection*{Acknowledgements}  Both authors are supported by the ERC-2014-CoG 646650 SingWave. P.R. would like to thank A. Soffer for stimulating discussions about this work and the Central China Normal University, Wuhan, where part of this work was done.


\section{Existence of
traveling waves}


This Section is devoted to the proof of the existence of solutions to
(\ref{tw3}). First, we want to reduce (\ref{tw4}) to a problem independent on
$\beta.$ For $1<s\leq2$ and $\beta\geq0$ we define the transform $\tau_{\beta
}$ by%
\begin{equation}
\left(  \tau_{\beta}u\right)  \left(  x\right)  :=\left(  \xi^{\ast}\right)
^{1/2}e^{i\left(  \xi^{\ast}x\right)  }u\left(  \xi^{\ast}x\right)  ,
\label{tw119}%
\end{equation}
where $\xi^{\ast}=\left(  \frac{2\beta}{s}\right)  ^{\frac{1}{s-1}}.$ For any
$\xi\in\mathbb{R}$ we define
\begin{equation}
n\left(  \xi\right)  =\left\vert \xi+1\right\vert ^{s}-s\xi-1 \label{tw114}%
\end{equation}
We consider the minimization problem%
\begin{equation}
I\left(  N\right)  =\inf\left\{  \mathcal{I}\left(  v\right)  :v\in
H^{s/2}\left(  \mathbb{R}\right)  \text{ and}\int\left\vert v\right\vert
^{2}=N\right\}  , \label{tw42}%
\end{equation}
where%
\[
\mathcal{I}\left(  v\right)  =\frac{1}{2}\left(  \int_{\mathbb{R}}{\overline
v}n\left(  D\right)  v-\frac{1}{s+1}\int_{\mathbb{R}}\left\vert v\right\vert
^{2s+2}\right)
\]
with~$n\left(  D\right)  =\mathcal{F}^{-1}n\left(  \xi\right)  \mathcal{F}$. A
minimizer of (\ref{tw42}) satisfies the equation
\begin{equation}
n\left(  D\right)  S+\eta S=S\left\vert S\right\vert ^{2s}, \label{tw103}%
\end{equation}
with some constant $\eta\in\mathbb{R}$. We now prove the following:

\begin{lemma}
\label{L1}Let $1<s\leq2$ and $\beta\geq0.$ Suppose that $S\in H^{s/2}\left(
\mathbb{R}\right)  $ is a minimizer for (\ref{tw42}). Then, $Q_{\beta}\left(
x\right)  =\left(  \tau_{\beta}S\right)  \left(  x\right)  $ minimizes
(\ref{tw4}). Moreover, if $S\in H^{s/2}\left(  \mathbb{R}\right)  $ solves
(\ref{tw103}) with some Lagrange multiplier $\theta\in\mathbb{R}$ , $Q_{\beta
}\left(  x\right)  =\left(  \tau_{\beta}S\right)  \left(  x\right)  $ is a
solution to (\ref{tw3}) with $\gamma=\tilde{\gamma}=\left(  \xi^{\ast}\right)
^{s}\left(  \eta+s-1\right)  .$
\end{lemma}

\begin{proof}
We denote
\[
m_{\beta}\left(  \xi\right)  :=\left\vert \xi\right\vert ^{s}-2\beta\xi.
\]
Observe that
\[
m_{\beta}^{\prime}\left(  \xi\right)  =0\text{ for }\xi=\xi^{\ast}%
\]
and%
\[
m_{\beta}\left(  \xi^{\ast}\right)  =-\left(  \xi^{\ast}\right)  ^{s}\left(
s-1\right)  .
\]
Note that
\[
\left(  \left\vert D\right\vert ^{s}-2\beta D-m_{\beta}\left(  \xi^{\ast
}\right)  \right)  \left(  e^{i\xi^{\ast}x}v\right)  =e^{i\xi^{\ast}x}\left(
\left\vert D+\xi^{\ast}\right\vert ^{s}-2\beta D-\left(  \xi^{\ast}\right)
^{s}\right)  v.
\]
Then, using that $Q_{\beta}\left(  x\right)  =\left(  \tau_{\beta}S\right)
\left(  x\right)  $ we have%
\begin{equation}
\left.
\begin{array}
[c]{c}%
\left[  \left\vert D\right\vert ^{s}-2\beta D\right]  Q_{\beta}=\left[
\left\vert D\right\vert ^{s}-2\beta D-m_{\beta}\left(  \xi^{\ast}\right)
\right]  Q_{\beta}+m_{\beta}\left(  \xi^{\ast}\right)  Q_{\beta}\\
=e^{i\xi^{\ast}x}\left(  \xi^{\ast}\right)  ^{1/2}\left(  \left(  \xi^{\ast
}\right)  ^{s}\left(  n\left(  D\right)  S\right)  \left(  \xi^{\ast}x\right)
+m_{\beta}\left(  \xi^{\ast}\right)  S\left(  \xi^{\ast}x\right)  \right)  .
\end{array}
\right.  \label{tw13}%
\end{equation}
Then, we get%
\[
\left.  \mathcal{E}_{\beta}\left(  Q_{\beta}\right)  =\left(  \xi^{\ast
}\right)  ^{s}\mathcal{I}\left(  S\right)  +\frac{m_{\beta}\left(  \xi^{\ast
}\right)  N}{2}.\right.
\]
Hence, if $S$ minimizers (\ref{tw42}), $Q_{\beta}$ solves (\ref{tw4}). Next,
consider equation (\ref{tw3})$.$ Using (\ref{tw13}) we obtain
\[
\left.
\begin{array}
[c]{c}%
\left[  \left\vert D\right\vert ^{s}-2\beta D+\gamma\right]  Q_{\beta
}-Q_{\beta}\left\vert Q_{\beta}\right\vert ^{2s}\\
=e^{i\left(  \xi^{\ast}x\right)  }\left(  \xi^{\ast}\right)  ^{\frac{2s+1}{2}%
}\left(  n\left(  D\right)  S-S\left\vert S\right\vert ^{2s}+\left(  \xi
^{\ast}\right)  ^{-s}\left[  \gamma-\left(  \xi^{\ast}\right)  ^{s}\left(
s-1\right)  \right]  S\right)  ,
\end{array}
\right.
\]
and thus, as $\gamma=\tilde{\gamma}$ and $S$ solves (\ref{tw103}), we conclude
that
\[
\left[  \left\vert D\right\vert ^{s}-2\beta D+\gamma\right]  Q_{\beta
}-Q_{\beta}\left\vert Q_{\beta}\right\vert ^{2s}=0.
\]
This proves Lemma \ref{L1}.
\end{proof}

Below we will show that problem (\ref{tw42}) has a minimizer.\ More precisely,
we aim to prove the following:

\begin{theorem}
\label{T1}Let $1<s<2.$ Then, for all $0<N<\left\langle Q,Q\right\rangle ,$
problem (\ref{tw42}) has a minimizer $S_{N}\in H^{s/2}\left(  \mathbb{R}%
\right)  $. In particular, $S_{N}$ solves (\ref{tw103}) with some $\eta
=\eta_{N}\in\mathbb{R}$.
\end{theorem}

\begin{remark}
Note that Theorem \ref{Th1} follows immediately from Lemma \ref{L1} and
Theorem \ref{T1}.
\end{remark}

We begin by preparing several results that are involved in the proof of
Theorem \ref{T1}. First we prove one elementary lemma.

\begin{lemma}
For any $0\leq A<1,$ the estimate%
\begin{equation}
n\left(  \xi\right)  -A\left\vert \xi\right\vert ^{s}\geq\frac{1}{2}\left(
1-A\right)  \left\vert \xi\right\vert ^{s}-C\left(  A\right)  , \label{tw14}%
\end{equation}
with some $C\left(  A\right)  >0\ $is satisfied.
\end{lemma}

\begin{proof}
Since%
\[
\left\vert \xi+1\right\vert ^{s}\geq\left\vert \xi\right\vert ^{s}%
-2^{s}\left\vert \xi\right\vert ^{s-1},\text{ for }\left\vert \xi\right\vert
\geq2,
\]
we have%
\[
n\left(  \xi\right)  -A\left\vert \xi\right\vert ^{s}\geq\left(  1-A\right)
\left\vert \xi\right\vert ^{s}-2^{s+2}\left\vert \xi\right\vert .
\]
Then, noting that
\[
\frac{1}{2}\left(  1-A\right)  \left\vert \xi\right\vert ^{s}-2^{s+2}%
\left\vert \xi\right\vert \geq0,
\]
if $\left\vert \xi\right\vert \geq c_{1}\left(  A\right)  =\left(
2^{s+3}\left(  1-A\right)  ^{-1}\right)  ^{\frac{1}{s-1}},$ we deduce that%
\begin{equation}
n\left(  \xi\right)  -A\left\vert \xi\right\vert ^{s}\geq\frac{1}{2}\left(
1-A\right)  \left\vert \xi\right\vert ^{s}, \label{tw111}%
\end{equation}
for $\left\vert \xi\right\vert \geq c_{1}\left(  A\right)  .$ If $\left\vert
\xi\right\vert \leq c_{1}\left(  A\right)  ,$ as $n\left(  \xi\right)  \geq0,$
for all $\xi\in\mathbb{R}$, we estimate%
\begin{equation}
n\left(  \xi\right)  -A\left\vert \xi\right\vert ^{s}\geq\frac{1}{2}\left(
1-A\right)  \left\vert \xi\right\vert ^{s}-c_{2}\left(  A\right)  ,
\label{tw110}%
\end{equation}
where $c_{2}\left(  A\right)  =\left(  A+1\right)  \left(  c_{1}\left(
A\right)  \right)  ^{s}.$ Estimates (\ref{tw111}) and (\ref{tw110}) imply
(\ref{tw14}).
\end{proof}

In order to prove Theorem \ref{T1}, first we show that $I\left(  N\right)  $
is bounded from below. We prove the following:

\begin{lemma}
\label{L4}Let $1<s\leq2.$ Then, for all $0<N\leq\left\langle Q,Q\right\rangle
,$ the following inequality holds%
\begin{equation}
2\mathcal{I}\left(  v\right)  \geq\frac{1}{2}\left(  1-\frac{N^{s}%
}{\left\langle Q,Q\right\rangle ^{s}}\right)  \int_{\mathbb{R}}{\overline
u}\left\vert D\right\vert ^{s}u-C\left(  N\right)  \geq-C\left(  N\right)  ,
\label{tw6}%
\end{equation}
for $v\in H^{s/2}\left(  \mathbb{R}\right)  $ such that $\int\left\vert
v\right\vert ^{2}=N.$ Moreover, any minimizing sequence for problem
(\ref{tw42}) is bounded in $H^{s/2}\left(  \mathbb{R}\right)  ,$ for
$0<N<\left\langle Q,Q\right\rangle .$
\end{lemma}

\begin{proof}
Let $v\in H^{s/2}\left(  \mathbb{R}\right)  $ be such that $\int\left\vert
v\right\vert ^{2}=N.$ By (\ref{tw5}) we have%
\begin{equation}
2\mathcal{I}\left(  v\right)  \geq\int_{\mathbb{R}}{\overline v}\left(
n\left(  D\right)  v-\frac{N^{s}}{\left\langle Q,Q\right\rangle ^{s}}%
\int_{\mathbb{R}}{\overline v}\left\vert D\right\vert ^{s}v\right)  .
\label{tw16}%
\end{equation}
Using (\ref{tw14}) with $A=\frac{N^{s}}{\left\langle Q,Q\right\rangle ^{s}}$
we get%
\begin{equation}
n\left(  \xi\right)  -\frac{N^{s}}{\left\langle Q,Q\right\rangle ^{s}%
}\left\vert \xi\right\vert ^{s}\geq\frac{1}{2}\left(  1-\frac{N^{s}%
}{\left\langle Q,Q\right\rangle ^{s}}\right)  \left\vert \xi\right\vert
^{s}-C\left(  N\right)  , \label{tw15}%
\end{equation}
for some $C\left(  N\right)  >0.$ Using (\ref{tw15}) in (\ref{tw16}) we attain
(\ref{tw6}). The boundedness of minimizing sequences for (\ref{tw42}) follows
immediately from (\ref{tw6}).
\end{proof}

Next, we show that the infimum $I\left(  N\right)  $ is strictly negative. We have:

\begin{lemma}
Let $1<s<2$ and $0<N<\left\langle Q,Q\right\rangle .$ Then, the estimate
\begin{equation}
I\left(  N\right)  <0 \label{tw10}%
\end{equation}
holds.
\end{lemma}

\begin{proof}
Let $v\in\mathcal{S}$ be such that $\widehat{v}\left(  \xi\right)  $ is
supported on $B:=\{\xi\in\mathbb{R}$ : $\xi>0\}$ and $\int\left\vert
v\right\vert ^{2}=N.$ We take $\chi\left(  t\right)  $ such that $\chi\left(
t\right)  =1,$ for $0\leq t\leq1$ and $\chi\left(  t\right)  =0,$ for $t>1.$
By Taylor's theorem, for $\xi>0$ we have
\[
\left.
\begin{array}
[c]{c}%
n\left(  \xi\right)  =\chi\left(  \xi\right)  n\left(  \xi\right)  +\left(
1-\chi\left(  \xi\right)  \right)  n\left(  \xi\right) \\
=\chi\left(  \xi\right)  \left(
{\displaystyle\int_{0}^{\xi}}
n^{\prime\prime}\left(  \tau\right)  \left(  \xi-\tau\right)  dt\right)
+\left(  1-\chi\left(  \xi\right)  \right)  n\left(  \xi\right)  .
\end{array}
\right.
\]
Then, we get%
\begin{align*}
&
{\displaystyle\int_{\mathbb{R}}}
n\left(  \xi\right)  \left\vert \hat{v}\left(  \xi\right)  \right\vert
^{2}d\xi\\
&  =%
{\displaystyle\int_{\mathbb{R}}}
\chi\left(  \xi\right)  \left(
{\displaystyle\int\nolimits_{0}^{\xi}}
n^{\prime\prime}\left(  \tau\right)  \left(  \xi-\tau\right)  d\tau\right)
\left\vert \hat{v}\left(  \xi\right)  \right\vert ^{2}d\xi+%
{\displaystyle\int_{\mathbb{R}}}
\left(  1-\chi\left(  \xi\right)  \right)  \left(  \xi+1\right)
^{s}\left\vert \hat{v}\left(  \xi\right)  \right\vert ^{2}d\xi\\
&  -%
{\displaystyle\int_{\mathbb{R}}}
\left(  1-\chi\left(  \xi\right)  \right)  \left(  s\xi+1\right)  \left\vert
\hat{v}\left(  \xi\right)  \right\vert ^{2}d\xi.
\end{align*}
Thus, using that%
\[
\chi\left(  \xi\right)
{\displaystyle\int\nolimits_{0}^{\xi}}
n^{\prime\prime}\left(  \tau\right)  \left(  \xi-\tau\right)  d\tau\leq
\frac{s\left(  s-1\right)  }{2}\xi^{2}%
\]
and$\allowbreak$%
\[
\left(  1-\chi\left(  \xi\right)  \right)  \left(  \xi+1\right)  ^{s}\leq
2^{s}\xi^{2},
\]
for $\xi>0,$ we see that%
\[%
{\displaystyle\int_{\mathbb{R}}}
n\left(  \xi\right)  \left\vert \hat{v}\left(  \xi\right)  \right\vert
^{2}d\xi\leq2^{s+1}\int_{\mathbb{R}}\xi^{2}\left\vert \hat{v}\left(
\xi\right)  \right\vert ^{2}d\xi.
\]
Therefore, we obtain%
\[
\mathcal{I}\left(  v\right)  \leq2^{s}\int_{\mathbb{R}}{\overline{v}}\left(
-\Delta\right)  v-\frac{1}{2s+2}\int_{\mathbb{R}}\left\vert v\right\vert
^{2s+2}.
\]
Let now $w\in\mathcal{S}$ be such that $\widehat{w}\left(  \xi\right)  $ is
supported on $B$ and $\int\left\vert w\right\vert ^{2}=N.$ We take $v\left(
x\right)  =\lambda^{1/2}w\left(  \lambda x\right)  .$ Then,%
\begin{equation}
\mathcal{I}\left(  v\right)  \leq\lambda^{s}\left(  2^{s}\lambda^{2-s}%
\int_{\mathbb{R}}\overline{w}\left(  -\Delta\right)  w-\frac{1}{2s+2}%
\int_{\mathbb{R}}\left\vert w\right\vert ^{2s+2}\right)  . \label{tw112}%
\end{equation}
As $s<2,$ choosing $\lambda>0$ small enough we show that
\begin{equation}
2^{s}\lambda^{2-s}\int_{\mathbb{R}}\overline{w}\left(  -\Delta\right)
w-\frac{1}{2s+2}\int_{\mathbb{R}}\left\vert w\right\vert ^{2s+2}<0.
\label{tw113}%
\end{equation}
Therefore, by (\ref{tw112}) and (\ref{tw113}) we obtain
\[
I\left(  N\right)  \leq\mathcal{I}\left(  v\right)  <0
\]
and (\ref{tw10}) follows.
\end{proof}

In the next lemma we show that $I\left(  N\right)  $ enjoys a strict
sub-additivity condition.

\begin{lemma}
\label{L5}Let $1<s<2$, $0<N<\left\langle Q,Q\right\rangle $ and $0<\alpha<N.$
Then, the following estimate holds%
\begin{equation}
I\left(  N\right)  <I\left(  \alpha\right)  +I\left(  N-\alpha\right)  .
\label{tw18}%
\end{equation}
Moreover, the function $I\left(  N\right)  $ is strictly decreasing and
continuous on $0<N<\left\langle Q,Q\right\rangle .$
\end{lemma}

\begin{proof}
We follow the proof of Lemma 2.3 of \cite{Frohlich}. Suppose that
$0<N<\left\langle Q,Q\right\rangle .$ Then, by Lemma \ref{L4}, $I\left(
N\right)  $ is finite. Observe that%
\[
I\left(  N\right)  =N\left(  I_{1}\left(  N\right)  \right)  ,
\]
where%
\begin{equation}
I_{1}\left(  N\right)  :=\inf_{v\in H^{s/2}\left(  \mathbb{R}\right)  ,\text{
}\left\Vert v\right\Vert _{L^{2}}^{2}=1}\left\{  \frac{1}{2}\left(
\int_{\mathbb{R}}{\overline v}n\left(  D\right)  v-\frac{N^{s}}{s+1}%
\int_{\mathbb{R}}\left\vert v\right\vert ^{2s+2}\right)  \right\}  .
\label{tw12}%
\end{equation}
We can restrict the infimum in (\ref{tw12}) to elements $v\in H^{s/2}\left(
\mathbb{R}\right)  ,$ $\left\Vert v\right\Vert _{L^{2}}^{2}=1$ such that
\begin{equation}
\int_{\mathbb{R}}\left\vert v\right\vert ^{2s+2}\geq c>0. \label{tw17}%
\end{equation}
Indeed, otherwise there exists a minimizing sequence $\{v_{m}\}_{m=0}^{\infty
}$ such that%
\[
\int_{\mathbb{R}}\left\vert v_{m}\right\vert ^{2s+2}\rightarrow0.
\]
As $n\left(  \xi\right)  \geq0,$ for all $\xi\in\mathbb{R}$, we see that%
\[
\int_{\mathbb{R}}{\overline v}_{m}n\left(  D\right)  v_{m}\geq0,
\]
for all $m\geq1.$ Then,
\[
I\left(  N\right)  =\mathop{\rm lim}_{m\rightarrow\infty}\mathcal{I}\left(
v_{m}\right)  \geq0.
\]
The last relation contradicts (\ref{tw10}). Hence, (\ref{tw17}) holds.
Relation (\ref{tw17}) implies that $I_{1}\left(  N\right)  ,$ as function of
$N,$ is strictly decreasing. Then,%
\[
I\left(  \theta N\right)  =\theta N\left(  I_{1}\left(  \theta N\right)
\right)  <\theta N\left(  I_{1}\left(  N\right)  \right)  =\theta I\left(
N\right)  ,
\]
for $\theta>1.$ Thus, (\ref{tw18}) follows from Lemma II.1 of \cite{Lions}.
Now, as $I_{1}\left(  N\right)  $ is strictly decreasing and $I_{1}\left(
N\right)  <0,$ by (\ref{tw10}), for $0<N<N_{1}<\left\langle Q,Q\right\rangle ,
$ we have
\[
I\left(  N_{1}\right)  =N_{1}\left(  I_{1}\left(  N_{1}\right)  \right)
<N_{1}\left(  I_{1}\left(  N\right)  \right)  <N\left(  I_{1}\left(  N\right)
\right)  =I\left(  N\right)  ,
\]
and hence, $I\left(  N\right)  $ is strictly decreasing. Since $N^{s}$ is
convex for $s>1,$ $I_{1}\left(  N\right)  $ must be concave on
$0<N<\left\langle Q,Q\right\rangle $, and hence $I_{1}\left(  N\right)  \in
C\left(  \left(  0,\left\langle Q,Q\right\rangle \right)  \right)  .$
Therefore, it follows that $I\left(  N\right)  $ is continuous on
$0<N<\left\langle Q,Q\right\rangle .$
\end{proof}

We define the functional
\begin{equation}
\mathcal{L}\left(  u\right)  :=\int_{\mathbb{R}}\overline{u}n\left(  D\right)
u+\int_{\mathbb{R}}\left\vert u\right\vert ^{2}. \label{tw26}%
\end{equation}
We need now the following profile decomposition result for a bounded sequence
in $H^{s/2}\left(  \mathbb{R}\right)  .$

\begin{lemma}
\label{L2}Let $\left\{  u_{n}\right\}  _{n=1}^{\infty}$ be a bounded sequence
in $H^{s/2}\left(  \mathbb{R}\right)  ,$ $1<s<2.$ Then, there exist a
subsequence of $\left\{  u_{n}\right\}  _{n=1}^{\infty}$ (still denoted
$\left\{  u_{n}\right\}  _{n=1}^{\infty}$), a family $\{\mathbf{x}^{j}%
\}_{j=1}^{\infty}$ of sequences in $\mathbb{R},$ with%
\begin{equation}
\left\vert x_{n}^{k}-x_{n}^{j}\right\vert \rightarrow+\infty,\text{ as
}n\rightarrow\infty,\text{ for }k\neq j, \label{tw27}%
\end{equation}
and a sequence $\left\{  V^{j}\right\}  _{j=1}^{\infty}$ of $H^{s/2}\left(
\mathbb{R}\right)  $ functions, such that for every $l\geq1$ and every
$x\in\mathbb{R}$,%
\begin{equation}
u_{n}\left(  x\right)  =\sum_{j=1}^{l}V^{j}\left(  x-x_{n}^{j}\right)
+u_{n}^{l}\left(  x\right)  , \label{tw19}%
\end{equation}
where the series $\sum_{j=1}^{\infty}\left\Vert V^{j}\right\Vert
_{H^{s/2}\left(  \mathbb{R}\right)  }^{2}$ converges and
\begin{equation}
\limsup_{n\rightarrow\infty}\left\Vert u_{n}^{l}\left(  x\right)  \right\Vert
_{L^{p}\left(  \mathbb{R}\right)  }\rightarrow0,\text{ \ as }l\rightarrow
\infty, \label{tw28}%
\end{equation}
for every $p>2.$ Moreover, as $n\rightarrow\infty,$%
\begin{equation}
\left\Vert u_{n}\right\Vert _{L^{2}}^{2}=\sum_{j=1}^{l}\left\Vert
V^{j}\right\Vert _{L^{2}}^{2}+\left\Vert u_{n}^{l}\right\Vert _{L^{2}}%
^{2}+o\left(  1\right)  , \label{tw20}%
\end{equation}%
\begin{equation}
\left\Vert \left\vert \nabla\right\vert ^{s/2}u_{n}\right\Vert _{L^{2}}%
^{2}=\sum_{j=1}^{l}\left\Vert \left\vert \nabla\right\vert ^{s/2}%
V^{j}\right\Vert _{L^{2}}^{2}+\left\Vert \left\vert \nabla\right\vert
^{s/2}u_{n}^{l}\right\Vert _{L^{2}}^{2}+o\left(  1\right)  , \label{tw21}%
\end{equation}
and%
\begin{equation}
\mathcal{L}\left(  u_{n}\right)  =\sum_{j=1}^{l}\mathcal{L}\left(
V^{j}\right)  +\mathcal{L}\left(  u_{n}^{l}\right)  +o\left(  1\right)  .
\label{tw22}%
\end{equation}

\end{lemma}

\begin{proof}
See the proof of Proposition 2.1 of \cite{HmidiKeraani} for the case of
bounded sequence in $H^{1}\left(  \mathbb{R}\right)  .$ The proof for the
fractional case $H^{s/2}\left(  \mathbb{R}\right)  ,$ $1<s<2,$ is analogous.
We only make a comment on (\ref{tw22}). Since $n\left(  \xi\right)  \leq
C\left(  \left\vert \xi\right\vert ^{s}+1\right)  ,$ using (\ref{tw14}) with
$A=0,$ we deduce
\[
c\left(  \left\vert \xi\right\vert ^{s}+1\right)  \leq n\left(  \xi\right)
+1\leq C\left(  \left\vert \xi\right\vert ^{s}+1\right)  .
\]
for some constants $0<c\leq C.$ Therefore, the norm $\left\Vert u\right\Vert
_{\mathcal{L}}:=\sqrt{\mathcal{L}\left(  u\right)  }$ is equivalent to the
$H^{s/2}\left(  \mathbb{R}\right)  $ norm. Hence, relation (\ref{tw22}) is
obtained similarly to (\ref{tw20}) and (\ref{tw21}).
\end{proof}

\begin{lemma}
\label{L3}The functional $\mathcal{L}\left(  u\right)  ,$ defined by
(\ref{tw26}), is weakly semicontinuos in $H^{s/2}\left(  \mathbb{R}\right)  .$
That is, if $u_{n}\rightharpoonup u$ weakly in $H^{s/2}\left(  \mathbb{R}%
\right)  $, as $n\rightarrow\infty,$
\[
\liminf_{n\rightarrow\infty}\mathcal{L}\left(  u_{n}\right)  \geq
\mathcal{L}\left(  u\right)  .
\]
Moreover, if $\mathop{\rm lim}_{n\rightarrow\infty}\mathcal{L}\left(
u_{n}\right)  =\mathcal{L}\left(  u\right)  $ is valid, $\{u_{n}%
\}_{n=0}^{\infty}$ converges strongly to $u$ in $H^{s/2}\left(  \mathbb{R}%
\right)  .$
\end{lemma}

\begin{proof}
As noted in the proof of Lemma \ref{L2}, the norm $\left\Vert u\right\Vert
_{\mathcal{L}}=\sqrt{\mathcal{L}\left(  u\right)  }$ is equivalent to the
$H^{s/2}\left(  \mathbb{R}\right)  $ norm. By Plancharel theorem, the norm
$\left\Vert u\right\Vert _{\mathcal{L}}$ is equivalent to a weighted $L^{2}%
$-norm. Then, the result of Lemma \ref{L3} follows from Theorem 2.11 of
\cite{LiebLoss}.
\end{proof}

Now we have all the ingredients to prove Theorem \ref{T1}.

\begin{proof}
[Proof of Theorem \ref{T1}]Suppose that $0<N<\left\langle Q,Q\right\rangle ,$
where $Q$ is an optimizer of (\ref{tw5})$.$ Let $\{u_{n}\}_{n=0}^{\infty}$ be
a minimizing sequence for (\ref{tw42}). Then,
\begin{equation}
\mathop{\rm lim}_{n\rightarrow\infty}\mathcal{I}\left(  u_{n}\right)
=I\left(  N\right)  ,\text{ with }u_{n}\in H^{s/2}\left(  \mathbb{R}\right)
\text{ and }\int_{\mathbb{R}}\left\vert u_{n}\right\vert ^{2}=N,\text{ for all
}n\geq0. \label{tw32}%
\end{equation}
Since by Lemma \ref{L4} $\{u_{n}\}_{n=0}^{\infty}$ is bounded in
$H^{s/2}\left(  \mathbb{R}\right)  ,$ it follows from Lemma \ref{L2} that
there exist a subsequence of $\{u_{n}\}_{n=0}^{\infty}$ that we still denote
by $\{u_{n}\}_{n=0}^{\infty}$ and a sequence $\left\{  V^{j}\right\}
_{j=1}^{\infty}$ of $H^{s/2}\left(  \mathbb{R}\right)  $ functions, such that
for every $l\geq1$ and $x\in\mathbb{R}$, (\ref{tw19}) holds. Note that as
$\int_{\mathbb{R}}\left\vert u_{n}\right\vert ^{2}=N,$ by (\ref{tw20})
$\sum_{j=1}^{\infty}\left\Vert V^{j}\right\Vert _{L^{2}}\leq N.$ Moreover,
$\sum_{j=1}^{\infty}\left\Vert V^{j}\right\Vert _{L^{2}}$ is strictly
positive. Otherwise, $V^{j}\equiv0,$ for all $j\in\mathbb{N}$ and then,
(\ref{tw19}) would imply that $\left\Vert u_{n}\right\Vert _{L^{p}}%
\rightarrow0,$ as $n\rightarrow\infty,$ for all $p>2.$ This contradicts
(\ref{tw10}), by an argument similar to the proof of (\ref{tw17}).

Suppose that $\sum_{j=1}^{\infty}\left\Vert V^{j}\right\Vert _{L^{2}}=\alpha,$
for some $0<\alpha<N,$ and define%
\[
V_{n}\left(  x\right)  :=\sum_{j=1}^{l}V^{j}\left(  x-x_{n}^{j}\right)
\]
and%
\[
W_{n}\left(  x\right)  :=u_{n}^{l}\left(  x\right)  .
\]
Then, by (\ref{tw19}),%
\[
u_{n}\left(  x\right)  =V_{n}\left(  x\right)  +W_{n}\left(  x\right)  .
\]
Since
\begin{equation}
\left\Vert V_{n}\left(  x\right)  \right\Vert ^{2}=\sum_{j=1}^{l}\left\Vert
V^{j}\right\Vert ^{2}+\sum_{j,k=1,\text{ }j\neq k}^{l}\left(  V^{j}\left(
x-x_{n}^{j}\right)  ,V^{k}\left(  x-x_{n}^{k}\right)  \right)  \label{tw33}%
\end{equation}
and%
\[
\sum_{j,k=1,\text{ }j\neq k}^{l}\left(  V^{j}\left(  x-x_{n}^{j}\right)
,V^{k}\left(  x-x_{n}^{k}\right)  \right)  =o\left(  1\right)  ,\text{ as
}n\rightarrow\infty,
\]
due to the orthogonality property (\ref{tw27}), we have
\begin{equation}
\left\Vert V_{n}\left(  x\right)  \right\Vert ^{2}=\sum_{j=1}^{l}\left\Vert
V^{j}\right\Vert ^{2}+o\left(  1\right)  ,\text{ as }n\rightarrow\infty.
\label{tw34}%
\end{equation}
Thus, as $\sum_{j=1}^{\infty}\left\Vert V^{j}\right\Vert _{L^{2}}=\alpha,$
from (\ref{tw20}) we deduce that for a given $\varepsilon>0,$ there is some
$l,n_{0}\geq0,$ such that for all $n\geq n_{0},$%
\begin{equation}
\left\vert \left\Vert V_{n}\left(  x\right)  \right\Vert ^{2}-\alpha
\right\vert \leq\varepsilon\text{ and }\left\vert \left\Vert W_{n}\right\Vert
^{2}-\left(  N-\alpha\right)  \right\vert \leq\varepsilon. \label{tw36}%
\end{equation}
Using that $\left(  \left\vert a\right\vert +\left\vert b\right\vert \right)
^{p}\leq C\left(  \left\vert a\right\vert ^{p-1}\left\vert b\right\vert
+\left\vert a\right\vert \left\vert b\right\vert ^{p-1}\right)  ,$ we have%
\begin{align*}
&  \left\vert \left\vert V_{n}\left(  x\right)  +W_{n}\left(  x\right)
\right\vert ^{p}-\left\vert V_{n}\left(  x\right)  \right\vert ^{p}-\left\vert
W_{n}\left(  x\right)  \right\vert ^{p}\right\vert \\
&  \leq C\left(  \left\vert V_{n}\left(  x\right)  \right\vert ^{p-1}%
\left\vert W_{n}\left(  x\right)  \right\vert +\left\vert V_{n}\left(
x\right)  \right\vert \left\vert W_{n}\left(  x\right)  \right\vert
^{p-1}\right)  .
\end{align*}
Then, for some $l$ and all $n$ sufficiently large, by (\ref{tw27}) and
(\ref{tw28}) we deduce
\begin{equation}
\left\Vert u_{n}\right\Vert _{L^{p}}^{p}\leq\left\Vert V_{n}\left(  x\right)
\right\Vert _{L^{p}}^{p}+\left\Vert W_{n}\left(  x\right)  \right\Vert
_{L^{p}}^{p}+\delta_{p}\left(  \varepsilon\right)  , \label{tw30}%
\end{equation}
with $\delta_{p}\left(  \varepsilon\right)  \rightarrow0,$ as $\varepsilon
\rightarrow0.$ Similarly to (\ref{tw34}) we show that%
\[
\mathcal{L}\left(  V_{n}\right)  =\sum_{j=1}^{l}\mathcal{L}\left(
V^{j}\right)  +o\left(  1\right)  ,\text{ as }n\rightarrow\infty.
\]
Therefore, from (\ref{tw22}) we see that%
\begin{equation}
\mathcal{L}\left(  u_{n}\right)  =\mathcal{L}\left(  V_{n}\right)
+\mathcal{L}\left(  W_{n}\right)  +o\left(  1\right)  ,\text{ as }%
n\rightarrow\infty. \label{tw35}%
\end{equation}
Using (\ref{tw30}) and (\ref{tw35}) we obtain
\begin{equation}
I\left(  N\right)  =\mathop{\rm lim}_{n\rightarrow\infty}\mathcal{I}\left(
u_{n}\right)  \geq\liminf_{n\rightarrow\infty}\mathcal{I}\left(  V_{n}\right)
+\liminf_{n\rightarrow\infty}\mathcal{I}\left(  W_{n}\right)  -\delta
_{p}\left(  \varepsilon\right)  . \label{tw24}%
\end{equation}
Then, as $I\left(  N\right)  $ is strictly decreasing by Lemma \ref{L5}, using
(\ref{tw36}), from (\ref{tw24}) we deduce that
\begin{equation}
I\left(  N\right)  \geq I\left(  \alpha+\varepsilon\right)  +I\left(  \left(
N-\alpha\right)  +\varepsilon\right)  -\delta_{p}\left(  \varepsilon\right)  .
\label{tw25}%
\end{equation}
By Lemma \ref{L5}, $I\left(  N\right)  $ is continuous on $0<N<\left\langle
Q,Q\right\rangle $. Then, taking the limit as $\varepsilon\rightarrow0$ in
(\ref{tw25}) we see that
\begin{equation}
I\left(  N\right)  \geq I\left(  \alpha\right)  +I\left(  N-\alpha\right)  ,
\label{tw37}%
\end{equation}
for some $\alpha\in\left(  0,N\right)  .$ This contradicts (\ref{tw18}).
Therefore, we get $\sum_{j=1}^{\infty}\left\Vert V^{j}\right\Vert _{L^{2}}=N.$

Suppose that $\sum_{j=1}^{\infty}\left\Vert V^{j}\right\Vert _{L^{2}}=N$ and
that there are at least two functions $V^{j^{0}}$ and $V^{k}$, for some
$j^{0},k\geq0,$ that are not identically $0.$ Let $\alpha=\left\Vert V^{j^{0}%
}\right\Vert _{L^{2}}^{2}.$ By assumption, $0<\alpha<N.$ We define%
\[
\tilde{V}_{n}\left(  x\right)  :=V^{j^{0}}\left(  x-x_{n}^{j^{0}}\right)
\]
and
\[
\tilde{W}_{n}\left(  x\right)  :=\sum_{j=1,\text{ }j\neq j^{0}}^{l}%
V^{j}\left(  x-x_{n}^{j}\right)  +u_{n}^{l}\left(  x\right)  .
\]
Note that by (\ref{tw20}) $\limsup_{n\rightarrow\infty}\left\Vert u_{n}%
^{l}\right\Vert _{L^{2}}^{2}\rightarrow0,$ as $l\rightarrow\infty.$ Then,
similarly to (\ref{tw36}) we have%
\[
\left\vert \left\Vert \tilde{V}_{n}\left(  x\right)  \right\Vert ^{2}%
-\alpha\right\vert \leq\varepsilon\text{ and }\left\vert \left\Vert \tilde
{W}_{n}\right\Vert ^{2}-\left(  N-\alpha\right)  \right\vert \leq\varepsilon.
\]
Then, arguing as in the case when $\sum_{j=1}^{\infty}\left\Vert
V^{j}\right\Vert _{L^{2}}=\alpha,$ for some $0<\alpha<N,$ we arrive to
(\ref{tw37}), which contradicts (\ref{tw18}).

By discussion so far, the sum $\sum_{j=1}^{\infty}\left\Vert V^{j}\right\Vert
_{L^{2}}=\left\Vert V^{j^{0}}\right\Vert _{L^{2}}=N,$ for some $j^{0}\geq1.$
Then, Lemma \ref{L2} implies that there exist a subsequence of $\{u_{n}%
\}_{n=0}^{\infty}$ (still denoted by $\{u_{n}\}_{n=0}^{\infty}$) and a
sequence of real numbers $\{x_{n}\}_{n=0}^{\infty},$ such that $\tilde{u}%
_{n}:=u_{n}\left(  \cdot+x_{n}\right)  $ converges strongly in $L^{p}\left(
\mathbb{R}\right)  ,$ $p\geq2,$ to $S_{N}:=V^{j^{0}},$ as $n\rightarrow
\infty.$ Therefore, by Lemma \ref{L3} we have
\begin{equation}
I\left(  N\right)  =\mathop{\rm lim}_{n\rightarrow\infty}\mathcal{I}\left(
\tilde{u}_{n}\right)  \geq\mathcal{I}\left(  S_{N}\right)  \geq I\left(
N\right)  . \label{tw38}%
\end{equation}
The last relation implies that $S_{N}\in H^{s/2}\left(  \mathbb{R}\right)  $
is a minimizer for (\ref{tw4}). Finally, we note that by (\ref{tw38}),
$\mathop{\rm lim}_{n\rightarrow\infty}\mathcal{L}\left(  \tilde{u}_{n}\right)
=\mathcal{L}\left(  S_{N}\right)  .$ Then, Lemma \ref{L3} shows that in fact
$\tilde{u}_{n}$ converges strongly to $S_{N}$ in $H^{s/2}\left(
\mathbb{R}\right)  ,$ as $n\rightarrow\infty.$
\end{proof}

\section{Small-mass behavior
of traveling waves.}

We now investigate the structure of small mass solitary waves. Consider the
minimization problem%

\begin{equation}
I_{0}=\inf\left\{  \mathcal{I}_{0}\left(  v\right)  :v\in H^{1}\left(
\mathbb{R}\right)  \text{ and }\int\left\vert v\right\vert ^{2}=s_{0}\right\}
, \label{tw55}%
\end{equation}
where
\[
s_{0}:=\left(  \frac{s\left(  s-1\right)  }{2}\right)  ^{-\frac{1}{s}}%
\]
and%
\[
\mathcal{I}_{0}\left(  v\right)  =\frac{1}{2}%
{\displaystyle\int\limits_{\mathbb{R}}}
\overline{v}\left\vert D\right\vert ^{2}v-\dfrac{1}{2s+2}%
{\displaystyle\int\limits_{\mathbb{R}}}
\left\vert v\right\vert ^{2s+2}.
\]
We now formulate a result on existence and characterization of minimizers for
problem (\ref{tw55}) (see Chapter 8 of \cite{Cazenave}).

\begin{proposition}
\label{P1}There exists a real, positive and radially symmetric function
$\mathcal{R}\in H^{1}$ such that:

i) The set of minimizers of (\ref{tw55}) is characterized by the family
$e^{i\gamma_{0}}\mathcal{R}\left(  x-x_{0}\right)  ,$ $x_{0},\gamma_{0}%
\in\mathbb{R}$.

ii) For any sequence $\{v_{n}\}_{n=0}^{\infty}\in H^{1},$ such that
$\left\Vert v_{n}\right\Vert \rightarrow s_{0}^{1/2}$ and $\mathcal{I}%
_{0}\left(  v_{n}\right)  \rightarrow I_{0},$ as $n\rightarrow\infty,$ there
exist $x_{n},\gamma_{n}\in\mathbb{R}$ and a strictly increasing sequence
$\phi:\mathbb{N\rightarrow}\mathbb{N}$, with the property:%
\[
e^{i\gamma_{\phi\left(  n\right)  }}v_{\phi\left(  n\right)  }\left(
\cdot+x_{\phi\left(  n\right)  }\right)  \rightarrow\mathcal{R},\text{ in
}H^{1}.
\]

iii) $\mathcal{R}\in H^{1}$ is the unique positive, radial symmetric solution
of (\ref{tw75}).
\end{proposition}

We aim to compare the minimizers $S_{N}$ of (\ref{tw42}) with the minimizer
$\mathcal{R}$ of (\ref{tw55}). For this purpose, we consider the following
minimization problem%
\begin{equation}
Y\left(  N\right)  =\inf\left\{  \mathcal{Y}_{N}\left(  v\right)  :v\in
H^{s/2}\left(  \mathbb{R}\right)  \text{ and}\int\left\vert v\right\vert
^{2}=s_{0},\right\}  , \label{tw11}%
\end{equation}
where
\[
\mathcal{Y}_{N}\left(  v\right)  =\frac{1}{2}\left(  \int_{\mathbb{R}%
}{\overline v}n_{N}\left(  D\right)  v-\frac{1}{s+1}\int_{\mathbb{R}%
}\left\vert v\right\vert ^{2s+2}\right)
\]
with%
\begin{equation}
n_{N}\left(  D\right)  =\mathcal{F}^{-1}n_{N}\left(  \xi\right)
\mathcal{F}\text{, \ }n_{N}\left(  \xi\right)  :=\frac{2n\left(  N^{\frac
{s}{2-s}}\xi\right)  }{s\left(  s-1\right)  N^{\frac{2s}{2-s}}} \label{tw127}%
\end{equation}
and $n\left(  \xi\right)  $ is defined by (\ref{tw114}). Let
\begin{equation}
R_{N}\left(  x\right)  =s_{0}^{1/2}N^{-\frac{1}{2-s}}S_{N}\left(  \frac
{x}{N^{\frac{s}{2-s}}}\right)  . \label{tw116}%
\end{equation}
Note that
\[
\mathcal{Y}_{N}\left(  R_{N}\right)  =s_{0}^{s+1}N^{-\frac{2+s}{2-s}%
}\mathcal{I}\left(  S_{N}\right)
\]
and%
\[
\int\left\vert R_{N}\right\vert ^{2}=s_{0}N^{-1}\int\left\vert S_{N}%
\right\vert ^{2}.
\]
Then, as $S_{N}$ minimizes (\ref{tw42}), $R_{N}$ is a minimizer for
(\ref{tw11}). Moreover $R_{N}$ satisfies the equation
\begin{equation}
n_{N}\left(  D\right)  R_{N}+\theta_{N}R_{N}-\left\vert R_{N}\right\vert
^{2s}R_{N}=0, \label{tw115}%
\end{equation}
with some Lagrange multiplier $\theta_{N}\in\mathbb{R}$.

Let $0<N<\left\langle Q,Q\right\rangle ,$ where $Q$ is an optimizer of
(\ref{tw5})$.$ We denote by $\mathbf{R}_{N}$ the set of minimizers of the
problem (\ref{tw11})$,$ which is not empty by Theorem \ref{T1} and
(\ref{tw116}). Now, we prove that $R_{N}$ converges to $\mathcal{R}$, as $N$
tends to $0.$ Namely, we aim to prove the following:

\begin{theorem}
\label{T2}Let $0<N<\left\langle Q,Q\right\rangle $ and $R_{N}\in\mathbf{R}%
_{N}.$ Then, there exist $\tilde{x},\tilde{\gamma}\in\mathbb{R}$,
$\tilde{\gamma}=\tilde{\gamma}\left(  N\right)  $ and $\tilde{x}=\tilde
{x}\left(  N\right)  ,$ such that the relation%
\begin{equation}
\mathop{\rm lim}_{N\rightarrow0}\left\Vert e^{i\tilde{\gamma}}R_{N}\left(
\cdot+\tilde{x}\right)  -\mathcal{R}\right\Vert _{H^{r}}=0, \label{tw65}%
\end{equation}
for any $r\geq0.$
\end{theorem}

We prepare a lemma that is involved in the proof of Theorem \ref{T2}.

\begin{lemma}
\label{L6}\bigskip Let $0<\alpha<1$ and $b\geq2$. Then, for $\kappa
_{0}=cb^{-\frac{2}{\left(  2-s\right)  \alpha}},$ with some $c>0,$ and any
$0\leq\kappa\leq\kappa_{0},$%
\begin{equation}
n\left(  \kappa\xi\right)  =\frac{s\left(  s-1\right)  }{2}\left(  \kappa
\xi\right)  ^{2}+O\left(  \kappa^{3\left(  1-\alpha\right)  }\right)  ,
\label{tw61}%
\end{equation}
for all $\left\vert \xi\right\vert \leq\kappa^{-\alpha}.$ Moreover,
\begin{equation}
n\left(  \kappa\xi\right)  -b\kappa^{2-s}\left\vert \kappa\xi\right\vert
^{s}\geq0, \label{tw47}%
\end{equation}
holds, for all $\left\vert \xi\right\vert \geq\kappa^{-\alpha}.$
\end{lemma}

\begin{proof}
By using Taylor's theorem$,$ we have
\begin{equation}
n\left(  \kappa\xi\right)  =\frac{s\left(  s-1\right)  }{2}\left(  \kappa
\xi\right)  ^{2}+\frac{s\left(  s-1\right)  \left(  s-2\right)  }{2}%
{\displaystyle\int\limits_{0}^{\kappa\xi}}
\left(  y+1\right)  ^{s-3}\left(  \kappa\xi-y\right)  ^{2}dy, \label{tw134}%
\end{equation}
for all $\left\vert \kappa\xi\right\vert \leq\frac{1}{2}.$ Then, if
$\kappa\leq\left(  \frac{1}{2}\right)  ^{\frac{1}{1-\alpha}}$ we get
(\ref{tw61}). Next we prove (\ref{tw47}). If $\kappa\leq\kappa_{1}:=\left(
2b\right)  ^{-\frac{1}{2-s}},$ there exists a constant $K>0$ (independent of
$\kappa$) such that (\ref{tw47}) is true for all $\left\vert \kappa
\xi\right\vert \geq K.$ Note now that $n\left(  \kappa\xi\right)  \geq0$ for
all $\xi\in\mathbb{R}$ , $n^{\prime}\left(  \kappa\xi\right)  <0,$ for
$\xi<0,$ and $n^{\prime}\left(  \kappa\xi\right)  >0,$ for $\xi>0.$ Then,
$\min_{\left\vert \kappa\xi\right\vert \geq\frac{1}{2}}n\left(  \kappa
\xi\right)  =c>0$ and thus, if $K\geq\frac{1}{2},$ for $\kappa_{2}=\left(
\frac{c}{bK^{s}}\right)  ^{\frac{1}{2-s}}$ and all $0<\kappa\leq\kappa_{2}, $
(\ref{tw47}) is satisfied on $\frac{1}{2}\leq\left\vert \kappa\xi\right\vert
\leq K.$ Suppose now that $\xi\in\mathbb{R}$ is such that $\kappa^{1-\delta
}\leq\left\vert \kappa\xi\right\vert \leq\min\{\frac{1}{2},\kappa
^{1-\frac{2+s}{2s}\delta}\},$ for all $\kappa\leq\left(  \frac{1}{2}\right)
^{\frac{1}{1-\delta}}$ and some $0<\delta<1.$ Then, using Taylor's theorem we
estimate
\begin{equation}
\left.
\begin{array}
[c]{c}%
n\left(  \kappa\xi\right)  -b\lambda^{2-s}\left\vert \lambda\xi\right\vert
^{s}\\
=s\left(  s-1\right)
{\displaystyle\int\limits_{0}^{\kappa\xi}}
\left(  y+1\right)  ^{s-2}\left(  \kappa\xi-y\right)  dy-b\kappa
^{2-s}\left\vert \lambda\xi\right\vert ^{s}\\
\geq\frac{s\left(  s-1\right)  }{2}\left(  \min\limits_{y\in\lbrack-\left\vert
\kappa\xi\right\vert ,\left\vert \kappa\xi\right\vert ]}\left(  y+1\right)
^{s-2}\right)  \left\vert \kappa\xi\right\vert ^{2}-b\kappa^{2-s}\left\vert
\kappa\xi\right\vert ^{s}\\
\geq\kappa^{2-2\delta}\left(  \frac{s\left(  s-1\right)  }{2^{3-s}}%
-b\kappa^{\frac{\left(  2-s\right)  \delta}{2}}\right)  .
\end{array}
\right.  \label{tw98}%
\end{equation}
Therefore, if $\kappa\leq\kappa_{3}^{\left(  \delta\right)  }:=\min\left\{
\left(  \frac{1}{2}\right)  ^{\frac{1}{1-\delta}},\left(  \frac{s\left(
s-1\right)  }{2^{3-s}b}\right)  ^{\frac{2}{\left(  2-s\right)  \delta}%
}\right\}  ,$ (\ref{tw47}) is satisfied on $\kappa^{1-\delta}\leq\left\vert
\kappa\xi\right\vert \leq\kappa^{1-\frac{2+s}{2s}\delta}.$ Let $m$ be such
that $\left(  \frac{2+s}{2s}\right)  ^{m}\alpha<1\leq\left(  \frac{2+s}%
{2s}\right)  ^{m+1}\alpha$. For $0\leq k\leq m,$ we define $\alpha
_{k}:=\left(  \frac{2+s}{2s}\right)  ^{k}\alpha.$ We decompose $[\kappa
^{1-\alpha},\frac{1}{2}]=\cup_{k=0}^{m-1}[\kappa^{1-\alpha_{k}},\kappa
^{1-\alpha_{k+1}}]\cup\lbrack\kappa^{1-\alpha_{m}},\frac{1}{2}].$ Using
(\ref{tw98}) with $\delta=\alpha_{k},$ $0\leq k\leq m,$ we see that
(\ref{tw47}) is true on $\kappa^{1-\alpha}\leq\left\vert \kappa\xi\right\vert
\leq\frac{1}{2},$ if $\kappa\leq\kappa_{3}:=\kappa_{3}^{\left(  \alpha\right)
}.$ Hence, putting $\kappa_{0}=\min\{\kappa_{1},\kappa_{2},\kappa_{3}\},$ we
conclude that (\ref{tw47}) is satisfied for all $\left\vert \xi\right\vert
\geq\kappa^{-\alpha},$ with $0<\kappa\leq\kappa_{0}.$
\end{proof}

\begin{proof}
[Proof of Theorem \ref{T2}]To somehow simplify the notation, we introduce
$\kappa=N^{\frac{s}{2-s}}$ and study the limit of
\[
R^{\left(  \kappa\right)  }=R_{\kappa^{\frac{2-s}{s}}},
\]
as $\kappa\rightarrow0.$ We denote $Y_{\kappa}=Y\left(  N\right)  $ and
$\mathcal{Y}^{\left(  \kappa\right)  }\left(  v\right)  =\mathcal{Y}%
_{N}\left(  v\right)  .$ Let $\sigma_{\kappa}\in L^{\infty}$ be such that
$\sigma_{\kappa}\left(  \xi\right)  =s_{0}^{1/2}\left(  \int_{\left\vert
\zeta\right\vert \leq\kappa^{-\alpha}}\left\vert \mathcal{\hat{R}}\left(
\zeta\right)  \right\vert ^{2}d\zeta\right)  ^{-1/2},$ for $\left\vert
\xi\right\vert \leq\kappa^{-\alpha},$ and $\sigma_{\kappa}\left(  \xi\right)
=0,$ for $\left\vert \xi\right\vert \geq\kappa^{-\alpha}.$ We put
$\mathcal{R}_{\kappa}:=\mathcal{F}^{-1}\left(  \sigma_{\kappa}\mathcal{\hat
{R}}\right)  .$ Note that $\int\left\vert \mathcal{\hat{R}}\right\vert
^{2}=s_{0}.$ Using (\ref{tw61}), we have%
\begin{equation}
\left.
\begin{array}
[c]{c}%
Y_{\kappa}\leq\mathcal{Y}^{\left(  \kappa\right)  }\left(  \mathcal{R}\right)
=\kappa^{-2}%
{\displaystyle\int_{\mathbb{R}}}
\frac{n\left(  \kappa\xi\right)  }{s\left(  s-1\right)  }\sigma_{\kappa
}\left(  \xi\right)  \left\vert \mathcal{\hat{R}}\left(  \xi\right)
\right\vert ^{2}d\xi-\dfrac{1}{2s+2}%
{\displaystyle\int}
\left\vert \mathcal{R}_{\kappa}\right\vert ^{2s+2}\\
\leq\frac{1}{2}%
{\displaystyle\int_{\mathbb{R}}}
\xi^{2}\sigma_{\kappa}\left(  \xi\right)  \left\vert \mathcal{\hat{R}}\left(
\xi\right)  \right\vert ^{2}d\xi-\dfrac{1}{2s+2}%
{\displaystyle\int}
\left\vert \mathcal{R}_{\kappa}\right\vert ^{2s+2}+O\left(  \kappa^{1-3\alpha
}\right)  .
\end{array}
\right.  \label{tw52}%
\end{equation}
Then, as
\[
\frac{1}{2}%
{\displaystyle\int_{\mathbb{R}}}
\xi^{2}\left(  1-\sigma_{\kappa}\left(  \xi\right)  \right)  \left\vert
\mathcal{\hat{R}}\left(  \xi\right)  \right\vert ^{2}d\xi+\dfrac{1}{2s+2}%
{\displaystyle\int}
\left(  \left\vert \mathcal{R}\right\vert ^{2s+2}-\left\vert \mathcal{R}%
_{\kappa}\right\vert ^{2s+2}\right)  =O\left(  \kappa\right)  ,
\]
we deduce
\begin{equation}
Y_{\kappa}\leq I_{0}+O\left(  \kappa^{1-3\alpha}\right)  . \label{tw54}%
\end{equation}
Let $\chi_{\kappa}\in L^{\infty}$ be such that $\chi_{\kappa}\left(
\xi\right)  =1,$ for $\left\vert \xi\right\vert \leq\kappa^{-1/3+\delta/3},$
$0<\delta<1,$ and $\chi_{\kappa}\left(  \xi\right)  =0,$ for $\left\vert
\xi\right\vert \geq\kappa^{-1/3+\delta/3}.$ We define
\begin{equation}
w_{\kappa}:=\mathcal{F}^{-1}\chi_{\kappa}\mathcal{F}R^{\left(  \kappa\right)
}\text{ \ and \ }r_{\kappa}:=R^{\left(  \kappa\right)  }-w_{\kappa}.
\label{tw59}%
\end{equation}
Let us prove that in fact $Y_{\kappa}\rightarrow I_{0},$ as $\kappa
\rightarrow0.$ Note that
\begin{equation}
Y_{\kappa}=\mathcal{Y}^{\left(  \kappa\right)  }\left(  R^{\left(
\kappa\right)  }\right)  =\mathcal{Y}^{\left(  \kappa\right)  }\left(
w_{\kappa}+r_{\kappa}\right)  . \label{tw73}%
\end{equation}
Using the elementary relation $\left(  a+b\right)  ^{p}\leq a^{p}%
+b^{p}+C\left(  a^{p-1}b+ab^{p-1}\right)  ,$ $a,b\geq0,$ $p\geq1,$ by Young's
inequality we have
\[
\left\vert w_{\kappa}+r_{\kappa}\right\vert ^{2s+2}\leq\left(  1+C\varepsilon
\right)  \left\vert w_{\kappa}\right\vert ^{2s+2}+\left(  1+K\left(
\varepsilon\right)  \right)  \left\vert r_{\kappa}\right\vert ^{2s+2},\text{
}\varepsilon>0,
\]
with $K\left(  \varepsilon\right)  :=C\varepsilon^{-\left(  \frac{1}%
{2s+1}+2s+1\right)  }.$ Using the last inequality in (\ref{tw73}) we get
\begin{equation}
\left.  Y_{\kappa}\geq Y_{1}\left(  \kappa\right)  +Y_{2}\left(
\kappa\right)  ,\right.  \label{tw50}%
\end{equation}
where%
\[
Y_{1}\left(  \kappa\right)  :=%
{\displaystyle\int\limits_{\mathbb{R}}}
\frac{n\left(  \kappa\xi\right)  }{s\left(  s-1\right)  \kappa^{2}}\left\vert
\hat{w}_{\kappa}\left(  \xi\right)  \right\vert ^{2}d\xi-\dfrac{1+C\varepsilon
}{2s+2}%
{\displaystyle\int\limits_{\mathbb{R}}}
\left\vert w_{\kappa}\right\vert ^{2s+2}%
\]
and%
\[
Y_{2}\left(  \kappa\right)  :=%
{\displaystyle\int\limits_{\mathbb{R}}}
\frac{n\left(  \kappa\xi\right)  }{s\left(  s-1\right)  \kappa^{2}}\left\vert
\hat{r}_{\kappa}\left(  \xi\right)  \right\vert ^{2}d\xi-\dfrac{1+K\left(
\varepsilon\right)  }{2s+2}%
{\displaystyle\int\limits_{\mathbb{R}}}
\left\vert r_{\kappa}\right\vert ^{2s+2}.
\]
By using (\ref{tw61}) with $\alpha=\frac{1-\delta}{3}$ we have%
\[
\left.  Y_{1}\left(  \kappa\right)  =\frac{1}{2}%
{\displaystyle\int\limits_{\mathbb{R}}}
\overline{w_{\kappa}}\left\vert D\right\vert ^{2}w_{\kappa}-\dfrac
{1+C\varepsilon}{2s+2}%
{\displaystyle\int\limits_{\mathbb{R}}}
\left\vert w_{\kappa}\right\vert ^{2s+2}+O\left(  \kappa^{\delta}\right)
{\displaystyle\int\limits_{\mathbb{R}}}
\left\vert \hat{w}_{\kappa}\left(  \xi\right)  \right\vert ^{2}d\xi,\right.
\]
as $\kappa\rightarrow0.$ Observe that%
\begin{equation}
\int\left\vert \hat{w}_{\kappa}\left(  \xi\right)  \right\vert ^{2}d\xi
+\int\left\vert \hat{r}_{\kappa}\left(  \xi\right)  \right\vert ^{2}d\xi
=\int\left\vert R^{\left(  \kappa\right)  }\right\vert ^{2}=s_{0}.
\label{tw45}%
\end{equation}
Hence, we show that
\begin{equation}
Y_{1}\left(  \kappa\right)  =\frac{1}{2}%
{\displaystyle\int\limits_{\mathbb{R}}}
\overline{w_{\kappa}}\left\vert D\right\vert ^{2}w_{\kappa}-\dfrac
{1+C\varepsilon}{2s+2}%
{\displaystyle\int\limits_{\mathbb{R}}}
\left\vert w_{\kappa}\right\vert ^{2s+2}+O\left(  \kappa^{\delta}\right)  ,
\label{tw49}%
\end{equation}
as $\kappa\rightarrow0.$ On the other hand, we claim that there is $\kappa
_{1}\left(  \varepsilon\right)  >0$ such that
\begin{equation}
Y_{2}\left(  \kappa\right)  \geq0, \label{tw48}%
\end{equation}
for any $0<\kappa\leq\kappa_{1}\left(  \varepsilon\right)  .$ Indeed, using
(\ref{tw5}) and (\ref{tw45})$,$ we estimate
\begin{equation}
\int_{\mathbb{R}}\left\vert r_{\kappa}\right\vert ^{2s+2}\leq C_{s}%
\int_{\mathbb{R}}\left\vert \xi\right\vert ^{s}\left\vert \hat{r}_{\kappa
}\left(  \xi\right)  \right\vert ^{2}d\xi. \label{tw44}%
\end{equation}
Then,%
\begin{equation}
Y_{2}\left(  \kappa\right)  \geq\frac{1}{s\left(  s-1\right)  \kappa^{2}}%
{\displaystyle\int\limits_{\mathbb{R}}}
\left(  n\left(  \kappa\xi\right)  -\dfrac{C_{s}s\left(  s-1\right)  \left(
1+K\left(  \varepsilon\right)  \right)  }{2s+2}\kappa^{2-s}\left\vert
\kappa\xi\right\vert ^{s}\right)  \left\vert \hat{r}_{\kappa}\left(
\xi\right)  \right\vert ^{2}d\xi. \label{tw56}%
\end{equation}
Therefore, since $\hat{r}_{\kappa}\left(  \xi\right)  =0$ for $\left\vert
\xi\right\vert \leq\kappa^{-1/3+\delta/3},$ using (\ref{tw47}) of Lemma
\ref{L6} with $\alpha=\frac{1-\delta}{3}$ and $b=\dfrac{C_{s}s\left(
s-1\right)  \left(  1+K\left(  \varepsilon\right)  \right)  }{2s+2},$ we get
(\ref{tw48}). Since by Lemma \ref{L6} $\kappa=O\left(  b^{-\frac{2}{\left(
2-s\right)  \alpha}}\right)  ,$ we note that
\begin{equation}
\varepsilon=O\left(  \kappa^{\frac{\left(  2-s\right)  \alpha}{2\left(
\frac{1}{2s+1}+2s+1\right)  }}\right)  . \label{tw155}%
\end{equation}
Using (\ref{tw49}) and (\ref{tw48}) in (\ref{tw50}) we see%
\begin{equation}
Y_{\kappa}\geq Y_{1}\left(  \kappa\right)  +Y_{2}\left(  \kappa\right)  \geq
Y_{1}\left(  \kappa\right)  =\frac{1}{2}%
{\displaystyle\int\limits_{\mathbb{R}}}
\overline{w_{\kappa}}\left\vert D\right\vert ^{2}w_{\kappa}-\dfrac
{1+C\varepsilon}{2s+2}%
{\displaystyle\int\limits_{\mathbb{R}}}
\left\vert w_{\kappa}\right\vert ^{2s+2}+O\left(  \kappa^{\delta}\right)  ,
\label{tw71}%
\end{equation}
as $\kappa\rightarrow0.$ Thus, taking into account (\ref{tw54}), we get
\[
\frac{1}{2}%
{\displaystyle\int\limits_{\mathbb{R}}}
\overline{w_{\kappa}}\left\vert D\right\vert ^{2}w_{\kappa}-\dfrac
{1+C\varepsilon}{2s+2}%
{\displaystyle\int\limits_{\mathbb{R}}}
\left\vert w_{\kappa}\right\vert ^{2s+2}\leq C,
\]
for any $0<\kappa\leq\kappa_{2}\left(  \varepsilon\right)  ,$ and some
$\kappa_{2}\left(  \varepsilon\right)  >0.$ Then, from (\ref{tw5}) and
(\ref{tw45}) we get%
\[%
{\displaystyle\int\limits_{\mathbb{R}}}
\left(  \frac{1}{2}\left\vert \xi\right\vert ^{2}-\dfrac{C_{s}\left(
1+C\varepsilon\right)  }{2s+2}\left\vert \xi\right\vert ^{s}\right)
\left\vert \hat{w}_{\kappa}\left(  \xi\right)  \right\vert ^{2}d\xi\leq C.
\]
Therefore, as for some $c>0,$ $\frac{\left\vert \xi\right\vert ^{2}}{2}%
-\dfrac{C_{s}\left(  1+C\varepsilon\right)  }{2s+2}\left\vert \xi\right\vert
^{s}\geq\frac{\left\vert \xi\right\vert ^{2}}{4},$ for all $\left\vert
\xi\right\vert \geq c,$ by (\ref{tw45}) we obtain the estimate
\begin{equation}
\left\Vert w_{\kappa}\right\Vert _{H^{1}}\leq C, \label{tw72}%
\end{equation}
uniformly on $0<\kappa\leq\kappa_{2}\left(  \varepsilon\right)  .$ Moreover,
using (\ref{tw47}) we get
\begin{equation}
\left.
\begin{array}
[c]{c}%
Y_{2}\left(  \kappa\right)  \geq\frac{1}{s\left(  s-1\right)  \kappa^{2}}%
{\displaystyle\int\limits_{\mathbb{R}}}
\left(  n\left(  \kappa\xi\right)  -\dfrac{C_{s}s\left(  s-1\right)  \left(
1+K\left(  \varepsilon\right)  \right)  }{2s+2}\kappa^{2-s}\left\vert
\kappa\xi\right\vert ^{s}\right)  \left\vert \hat{r}_{\kappa}\left(
\xi\right)  \right\vert ^{2}d\xi\\
\geq%
{\displaystyle\int\limits_{\mathbb{R}}}
\left\vert \xi\right\vert ^{s}\left\vert \hat{r}_{\kappa}\left(  \xi\right)
\right\vert ^{2}d\xi.
\end{array}
\right.  \label{tw109}%
\end{equation}
Then, since the left-hand side of the last relation is bounded by $I_{0},$ due
to (\ref{tw54}) and (\ref{tw50}), we deduce%
\[
\left\Vert r_{\kappa}\right\Vert _{H^{s/2}}^{2}\leq C,
\]
for all $\kappa>0$ sufficiently small. In particular, we have
\begin{align*}
\kappa^{-\frac{\left(  1-\delta\right)  s}{3}}%
{\displaystyle\int\limits_{\mathbb{R}}}
\left\vert \hat{r}_{\kappa}\left(  \xi\right)  \right\vert ^{2}d\xi &
=\kappa^{-\frac{\left(  1-\delta\right)  s}{3}}%
{\displaystyle\int\limits_{\left\vert \xi\right\vert \geq\kappa^{-1/3+\delta
/3}}}
\left\vert \hat{r}_{\kappa}\left(  \xi\right)  \right\vert ^{2}d\xi\\
&  \leq%
{\displaystyle\int\limits_{\left\vert \xi\right\vert \geq\kappa^{-1/3+\delta
/3}}}
\left\vert \xi\right\vert ^{s}\left\vert \hat{r}_{\kappa}\left(  \xi\right)
\right\vert ^{2}d\xi\leq C,
\end{align*}
which implies%
\begin{equation}%
{\displaystyle\int\limits_{\mathbb{R}}}
\left\vert \hat{r}_{\kappa}\left(  \xi\right)  \right\vert ^{2}d\xi\leq
C\kappa^{\frac{\left(  1-\delta\right)  s}{3}}. \label{tw108}%
\end{equation}
Then, it follows from (\ref{tw45}) that
\begin{equation}
\int\left\vert w_{\kappa}\right\vert ^{2}\rightarrow s_{0},\text{ as }%
\kappa\rightarrow0. \label{tw67}%
\end{equation}
Then, returning to (\ref{tw71}) and using (\ref{tw155}) we get%
\[
Y_{\kappa}\geq Y_{1}\left(  \kappa\right)  +Y_{2}\left(  \kappa\right)  \geq
Y_{1}\left(  \kappa\right)  =\mathcal{I}_{0}\left(  w_{\kappa}\right)
+o\left(  \kappa^{\delta_{1}}\right)  ,
\]
for some $0<\delta_{1}\leq\delta.$ Therefore, using (\ref{tw54}), (\ref{tw72})
and (\ref{tw67}), we obtain%
\begin{equation}
\left.
\begin{array}
[c]{c}%
I_{0}+o\left(  \kappa^{\delta_{1}}\right)  \geq Y_{\kappa}\geq Y_{1}\left(
\kappa\right)  +Y_{2}\left(  \kappa\right)  \geq Y_{1}\left(  \kappa\right) \\
=\mathcal{I}_{0}\left(  w_{\kappa}\right)  +o\left(  \kappa^{\delta_{1}%
}\right)  \geq I_{0}+o\left(  \kappa^{\delta_{1}}\right)  .
\end{array}
\right.  \label{tw160}%
\end{equation}
In particular, this means
\begin{equation}
\mathop{\rm lim}_{\kappa\rightarrow0}\mathcal{I}_{0}\left(  w_{\kappa}\right)
=I_{0}\text{ and }\mathop{\rm lim}_{\kappa\rightarrow0}Y_{2}\left(
\kappa\right)  =0.\text{ } \label{tw57}%
\end{equation}
Then, from (\ref{tw109}) and (\ref{tw108}) we deduce that
\begin{equation}
\mathop{\rm lim}_{\kappa\rightarrow0}\left\Vert r_{\kappa}\right\Vert
_{H^{s/2}}^{2}=\mathop{\rm lim}_{\kappa\rightarrow0}\left(
{\displaystyle\int\limits_{\left\vert \xi\right\vert \geq\kappa^{-1/3+\delta
/3}}}
\left\vert \hat{r}_{\kappa}\left(  \xi\right)  \right\vert ^{2}d\xi+%
{\displaystyle\int\limits_{\mathbb{R}}}
\left\vert \xi\right\vert ^{s}\left\vert \hat{r}_{\kappa}\left(  \xi\right)
\right\vert ^{2}d\xi\right)  =0. \label{tw58}%
\end{equation}
Now, we have all the estimates that we need to prove Theorem \ref{T2}%
\textbf{.} We argue as follows. Let $\{N_{n}\}_{n=1}^{\infty}$ be such that
$0<N_{n}<\left\langle Q,Q\right\rangle $ and $N_{n}\rightarrow0,$ as
$n\rightarrow\infty.$ Then, $\kappa_{n}=N_{n}^{\frac{s}{2-s}}$ also tends to
$0.$ We consider the sequence $\{w_{n}\}_{n=1}^{\infty},$ where $w_{n}%
:=w_{\kappa_{n}}.$ Taking into account (\ref{tw67}) and (\ref{tw57}), from
Proposition \ref{P1} it follows that there exist a subsequence of
$\{w_{n}\}_{n=1}^{\infty},$ that we still denote by $\{w_{n}\}_{n=1}^{\infty}$
and $x_{n},\gamma_{n}\in\mathbb{R}$, such that%
\[
\mathop{\rm lim}_{n\rightarrow\infty}\left\Vert e^{i\gamma_{n}}w_{n}\left(
\cdot+x_{n}\right)  -\mathcal{R}\right\Vert _{H^{1}}=0.
\]
Furthermore, using (\ref{tw59}) and (\ref{tw58}) we deduce that
\begin{equation}
\mathop{\rm lim}_{n\rightarrow\infty}\left\Vert e^{i\gamma_{n}}R^{\left(
\kappa_{n}\right)  }\left(  \cdot+x_{n}\right)  -\mathcal{R}\right\Vert
_{H^{s/2}}=0. \label{tw64}%
\end{equation}
To complete the proof, we need to show that $e^{i\gamma_{n}}R^{\left(
\kappa_{n}\right)  }\left(  \cdot+x_{n}\right)  $ converges to $\mathcal{R}$
in $H^{r},$ for all $r\geq0.$ Recall that for any $0<\kappa<\left\langle
Q,Q\right\rangle ^{\frac{2-s}{s}},$ $R^{\left(  \kappa\right)  }$ is a
minimizer of (\ref{tw11}). Then, $R^{\left(  \kappa\right)  }$ satisfies
equation (\ref{tw115}), that is%
\begin{equation}
n^{\left(  \kappa\right)  }\left(  D\right)  R^{\left(  \kappa_{n}\right)
}+\theta^{\left(  \kappa\right)  }R^{\left(  \kappa_{n}\right)  }-\left\vert
R^{\left(  \kappa_{n}\right)  }\right\vert ^{2s}R^{\left(  \kappa_{n}\right)
}=0, \label{tw60}%
\end{equation}
where%
\begin{equation}
n^{\left(  \kappa\right)  }\left(  D\right)  :=n_{\kappa^{\frac{2-s}{s}}%
}\left(  D\right)  , \label{tw74}%
\end{equation}
for some $\theta^{\left(  \kappa\right)  }=\theta_{\kappa^{\frac{2-s}{s}}}%
\in\mathbb{R}$. Since $R^{\left(  \kappa_{n}\right)  }$ converges to
$\mathcal{R},$ in $H^{s/2},$ as $n\rightarrow\infty,$ $\left\Vert R^{\left(
\kappa_{n}\right)  }\right\Vert _{H^{s/2}}\leq C,$ uniformly for
$n\in\mathbb{N}$. In particular, from Sobolev theorem we get $\int\left\vert
R^{\left(  \kappa_{n}\right)  }\right\vert ^{2s+2}\leq C,$ and then, using the
bound (\ref{tw54}) we see that
\[%
{\displaystyle\int\limits_{\mathbb{R}}}
n^{\left(  \kappa_{n}\right)  }\left(  \xi\right)  \left\vert \hat{R}^{\left(
\kappa_{n}\right)  }\left(  \xi\right)  \right\vert ^{2}d\xi\leq C.
\]
Hence, from (\ref{tw60}) we get%
\begin{equation}
\left\vert \theta^{\left(  \kappa_{n}\right)  }\right\vert \leq\int\left\vert
R^{\left(  \kappa_{n}\right)  }\right\vert ^{2s+2}+%
{\displaystyle\int\limits_{\mathbb{R}}}
n^{\left(  \kappa_{n}\right)  }\left(  \xi\right)  \left\vert \hat{R}^{\left(
\kappa_{n}\right)  }\left(  \xi\right)  \right\vert ^{2}d\xi\leq C,
\label{tw68}%
\end{equation}
uniformly for $n\in\mathbb{N}$. From (\ref{tw60}) and (\ref{tw68}), via the
Sobolev theorem, we also deduce%
\begin{equation}%
{\displaystyle\int\limits_{\mathbb{R}}}
\left(  n^{\left(  \kappa_{n}\right)  }\left(  \xi\right)  \right)
^{2}\left\vert \hat{R}^{\left(  \kappa_{n}\right)  }\left(  \xi\right)
\right\vert ^{2}d\xi\leq C\left(  \int\left\vert R^{\left(  \kappa_{n}\right)
}\right\vert ^{2}+\int\left\vert R^{\left(  \kappa_{n}\right)  }\right\vert
^{4s+2}\right)  \leq C. \label{tw62}%
\end{equation}
Then, it follows from Lemma \ref{L6} that $\left\Vert R^{\left(  \kappa
_{n}\right)  }\right\Vert _{H^{s}}\leq C.$ Applying the operator $D$ to
equation (\ref{tw60}) and arguing similarly to the proof of (\ref{tw62}), we
get $\left\Vert R^{\left(  \kappa_{n}\right)  }\right\Vert _{H^{s+1}}\leq C.$
By induction on $r\in\mathbb{N}$, we see that, in fact, $R^{\left(  \kappa
_{n}\right)  }$ is uniformly bounded in $H^{r},$ for all $r\in\mathbb{N}$ (and
hence all $r\geq0$). Therefore, as (\ref{tw64}) is true, $e^{i\gamma_{n}%
}R^{\left(  \kappa_{n}\right)  }\left(  \cdot+x_{n}\right)  $ converges to
$\mathcal{R}$ in $H^{r},$ $r\geq0,$ as $n\rightarrow\infty.$ Recalling that
$R^{\left(  \kappa_{n}\right)  }=R_{\kappa_{n}^{\frac{2-s}{s}}}$ and
$\kappa_{n}=N_{n}^{\frac{s}{2-s}}$ we deduce that for any $\{N_{n}%
\}_{n=1}^{\infty}$ such that $0<N_{n}<\left\langle Q,Q\right\rangle $ and
$N_{n}\rightarrow0,$ as $n\rightarrow\infty$, the estimate
\[
\mathop{\rm lim}_{n\rightarrow\infty}\left\Vert e^{i\gamma_{n}}R_{N_{n}%
}\left(  \cdot+x_{n}\right)  -\mathcal{R}\right\Vert _{H^{r}}=0,
\]
holds for any $r\geq0.$ Since the sequence $\{N_{n}\}_{n=1}^{\infty}$ in the
last relation is arbitrary, we attain (\ref{tw65}).
\end{proof}

Theorem \ref{T2} allows to calculate the limit as the mass $0<N<\left\langle
Q,Q\right\rangle $ tends to $0$ for Lagrange multiplier $\theta_{N}$ in the
equation (\ref{tw115}) for $R_{N}.$ We have the following.

\begin{lemma}
Let $0<N<\left\langle Q,Q\right\rangle .$ The relation%
\begin{equation}
\left\vert \theta_{N}-\lambda\left(  s\right)  \right\vert =o\left(  1\right)
,\text{ as }N\rightarrow0, \label{tw102}%
\end{equation}
is true, where $\lambda\left(  s\right)  =\left(  \rho_{0}^{s}\frac{s\left(
s-1\right)  }{2}\right)  ^{-\frac{2}{2-s}}$ is the Lagrange multiplier in
equation (\ref{tw75}) for $\mathcal{R}$.
\end{lemma}

\begin{proof}
Theorem \ref{T2} implies that there is $R_{N}\left(  x\right)  $, such that
$R_{N}$ converges to $\mathcal{R},$ in $H^{r},$ $r\geq0,$ as $N\rightarrow0. $
Since $R_{N}$ and $\mathcal{R}$ solve (\ref{tw115}) and (\ref{tw75}),
respectively, we have%
\begin{equation}
\left.
\begin{array}
[c]{c}%
\left\vert \theta_{N}-\lambda\left(  s\right)  \right\vert \mathcal{R}%
\leq\mathbf{r}+\left\vert n_{N}\left(  D\right)  \left(  \mathcal{R}%
-R_{N}\right)  \right\vert \\
+\left\vert \left\vert R_{N}\right\vert ^{2s}R_{N}-\mathcal{R}^{2s+1}%
\right\vert +\left\vert \theta_{N}\right\vert \left\vert \mathcal{R-}%
R_{N}\right\vert ,
\end{array}
\right.  \label{tw84}%
\end{equation}
where
\[
\mathbf{r}:=\left\vert \left(  \left\vert D\right\vert ^{2}\mathcal{-}%
n_{N}\left(  D\right)  \right)  \mathcal{\mathcal{R}}\right\vert .
\]
As $n_{N}\left(  \xi\right)  \leq C\left(  1+\xi^{2}\right)  $ for all $\xi
\in\mathbb{R}$,
\begin{equation}
n_{N}\left(  D\right)  \leq C\left\langle D\right\rangle ^{2}. \label{tw86}%
\end{equation}
By noting that the sequence $\{\kappa_{n}\}_{n=1}^{\infty}$ in (\ref{tw68}) is
arbitrary, as $\theta^{\left(  \kappa_{n}\right)  }=\theta_{\kappa_{n}%
^{\frac{2-s}{s}}}$ we have%
\begin{equation}
\ \left\vert \theta_{N}\right\vert \leq C, \label{tw87}%
\end{equation}
uniformly for all $N>0$ small enough. Then, since $R_{N}$ converges to
$\mathcal{R}$, using (\ref{tw86}), (\ref{tw87}) and Sobolev theorem, we see
that the last three terms in the right hand side of (\ref{tw84}) tend to $0,$
as $N\rightarrow0.$ Now, note that Sobolev theorem implies%
\begin{equation}
\left.  \mathbf{r}^{2}\leq%
{\displaystyle\int\limits_{\left\vert \xi\right\vert \leq\kappa^{-\alpha}}}
\left\langle \xi\right\rangle \left\vert \left\vert \xi\right\vert ^{2}%
-n_{N}\left(  \xi\right)  \right\vert \left\vert \mathcal{\hat{R}}\left(
\xi\right)  \right\vert ^{2}d\xi+\mathbf{\tilde{r}},\right.  \label{tw88}%
\end{equation}
with $\kappa=N^{\frac{s}{2-s}},$ $0<\alpha<1/3$ and
\[
\mathbf{\tilde{r}}:=%
{\displaystyle\int\limits_{\left\vert \xi\right\vert \geq\kappa^{-\alpha}}}
\left\langle \xi\right\rangle \left\vert \left\vert \xi\right\vert ^{2}%
-n_{N}\left(  \xi\right)  \right\vert \left\vert \mathcal{\hat{R}}\left(
\xi\right)  \right\vert ^{2}d\xi.
\]
Using (\ref{tw86}) we have%
\[
\mathbf{\tilde{r}}\leq C%
{\displaystyle\int\limits_{\left\vert \xi\right\vert \geq\kappa^{-\alpha}}}
\left\langle \xi\right\rangle ^{4}\left\vert \mathcal{\hat{R}}\left(
\xi\right)  \right\vert ^{2}d\xi\leq C\kappa^{\alpha}\left\Vert \mathcal{\hat
{R}}\right\Vert _{H^{2}}^{2}.
\]
Then, using (\ref{tw61}) to estimate the first term in the right hand side of
(\ref{tw88}), we see that $\mathbf{r}=o\left(  1\right)  $, as $N\rightarrow
0.$ Hence, from (\ref{tw84}) we attain (\ref{tw102}).
\end{proof}

\begin{proof}
[Proof of Theorem \ref{Th2}]To prove the first part of Theorem \ref{Th2} we
note that the minimizers$\ S_{N}$ of (\ref{tw42}) and the minimizers $R_{N}$
of (\ref{tw11}) are related by (\ref{tw116}).\ Moreover, by Lemma \ref{L1} we
see that $S_{N}$ is related to $Q_{\beta,N}$ by the equation
\begin{equation}
Q_{\beta,N}\left(  x\right)  =\left(  \tau_{\beta}S_{N}\right)  \left(
x\right)  , \label{tw118}%
\end{equation}
with $\tau_{\beta}$ given by (\ref{tw119}). Then, $R_{N}$ and $Q_{\beta,N}$
are related by (\ref{tw117}), and hence, the first part of Theorem \ref{Th2}
is consequence of Theorem \ref{T2}. To prove the second part we observe that
if $R_{N}$ satisfies (\ref{tw115}), $S_{N}$ solves (\ref{tw103}) with
$\eta=\frac{s\left(  s-1\right)  }{2}\theta_{N}.$ On the other hand, it
follows from Lemma \ref{L1} that the Lagrange multiplier $\eta$ corresponding
to $S_{N}$ is related to the Lagrange multiplier $\gamma\left(  \beta
,N\right)  $ that corresponds to $Q_{\beta,N}$ by the formula $\gamma\left(
\beta,N\right)  =\left(  \frac{2\beta}{s}\right)  ^{\frac{s}{s-1}}\left(
\eta+s-1\right)  .$ Then,
\[
\theta_{N}=\frac{2}{s}\left(  \left(  s-1\right)  ^{-1}\left(  \frac{2\beta
}{s}\right)  ^{-\frac{s}{s-1}}\gamma\left(  \beta,N\right)  -1\right)  .
\]
Therefore, the second part of Theorem \ref{Th2} follows from (\ref{tw102}).
\end{proof}

\section{Uniqueness of
traveling waves.}

We now turn to the proof of Theorem \ref{Th3}.\ In view of relations
(\ref{tw116}) and (\ref{tw118}), we need to show the unicity of the minimizers
$R_{N}\in\mathbf{R}_{N}.$ We aim to prove the following.

\begin{lemma}
\bigskip\label{L7}There exists $N_{0}>0$ with the following property: given
$0<N<N_{0}$ and $R_{N},\tilde{R}_{N}\in\mathbf{R}_{N},$ there exist $\phi
,y\in\mathbb{R}$ such that%
\[
\tilde{R}_{N}\left(  x\right)  =e^{i\phi}R_{N}\left(  x-y\right)  .
\]

\end{lemma}

Before proving Lemma \ref{L7} we present a lemma that is involved in its
proof. Consider some fixed $R_{N}\in\mathbf{R}_{N}$ and recall that $R_{N}$
satisfies equation (\ref{tw115}). We define the linearized operator
$\mathcal{L}_{R_{N}}:H^{s/2}\rightarrow H^{-s/2}$ close to $R_{N}$ by
\[
\mathcal{L}_{R_{N}}f:=n_{N}\left(  D\right)  f+\theta_{N}f-\left(  s+1\right)
\left\vert R_{N}\right\vert ^{2s}f-s\left(  \left\vert R_{N}\right\vert
^{2s-2}R_{N}^{2}\right)  {\overline f},\text{ \ \ }f\in H^{s/2}.
\]
Recall the notation $\left(  f,g\right)  =\operatorname{Re}\int f\overline
{g}.$ We now prove the following invertibility result for $\mathcal{L}_{R_{N}%
}.$

\begin{lemma}
\label{L8}There exists $N_{0}>0,$ such that for all $0<N<N_{0}$ and all
$R_{N}\in\mathbf{R}_{N}$ the estimate
\begin{equation}
\left\Vert f\right\Vert _{H^{s/2}}\leq C\left(  \left\Vert \mathcal{L}_{R_{N}%
}f\right\Vert _{H^{-s/2}}+\left\vert \left(  f,iR_{N}\right)  \right\vert
+\left\vert \left(  f,\nabla R_{N}\right)  \right\vert \right)  \label{tw76}%
\end{equation}
is true for all $f\in H^{s/2}$. Moreover, for all $F\in H^{-s/2}$ with
$\left(  F,iR_{N}\right)  =\left(  F,\nabla R_{N}\right)  =0,$ the problem
\[
\left\{
\begin{array}
[c]{c}%
\mathcal{L}_{R_{N}}f=F,\text{ \ }f\in H^{s/2},\\
\left(  f,iR_{N}\right)  =\left(  f,\nabla R_{N}\right)  =0,
\end{array}
\right.
\]
has a unique solution $f\in H^{s/2},$ and
\begin{equation}
\left\Vert f\right\Vert _{H^{s/2}}\leq C\left\Vert F\right\Vert _{H^{-s/2}}
\label{tw97}%
\end{equation}
holds.
\end{lemma}

\begin{proof}
Let $\mathcal{L}:H^{1}\rightarrow H^{-1}$ be the linearized operator for the
equation (\ref{tw75}) around $\mathcal{R}:$%
\[
\mathcal{L}f:=\left\vert D\right\vert ^{2}f+\lambda\left(  s\right)  f-\left(
s+1\right)  \mathcal{R}^{2s}f-s\mathcal{R}^{2s}{\overline f},\text{ \ \ }f\in
H^{1}.
\]
We write $f\in H^{1}$ as $f=h+ig,$ with real $h\ $and $g.$ Then,%
\[
\mathcal{L}f=L_{+}h+iL_{-}g
\]
where
\[
L_{+}h:=\left\vert D\right\vert ^{2}h+\lambda\left(  s\right)  h-\left(
2s+1\right)  \mathcal{R}^{2s}h
\]
and%
\[
L_{-}g:=\left\vert D\right\vert ^{2}g+\lambda\left(  s\right)  g-\mathcal{R}%
^{2s}g.
\]
Thus,
\begin{equation}
\left(  \mathcal{L}f,f\right)  =\left(  L_{+}h,h\right)  +\left(
L_{-}g,g\right)  , \label{tw78}%
\end{equation}
for real functions $h,g\in H^{1}.$ It is known (see Lemmas 2.1 and 2.2 of
\cite{Nakanishi}) that $\ker L_{+}=\operatorname*{span}\{\nabla\mathcal{R}\},$
$\left.  L_{+}\right\vert _{\left\{  \mathcal{R}\right\}  ^{\perp}}\geq0$ and
$\left.  L_{-}\right\vert _{\left\{  \mathcal{R}\right\}  ^{\perp}}>0.$ Then,
from (\ref{tw78}) we get%
\[
\left(  \mathcal{L}f,f\right)  \geq c\left\Vert f\right\Vert _{H^{1}}^{2},
\]
for all $f\in H^{1}$ such that $\left\vert \left(  f,\mathcal{R}\right)
\right\vert +\left\vert \left(  f,i\mathcal{R}\right)  \right\vert +\left\vert
\left(  f,\nabla\mathcal{R}\right)  \right\vert =0.$ Moreover, using the
relation $L_{+}\mathcal{R}=-2s\left(  \left\vert D\right\vert ^{2}%
+\lambda\left(  s\right)  \right)  \mathcal{R},$ we see that
\begin{equation}
\left\Vert f\right\Vert _{H^{1}}\leq C\left(  \left\Vert \mathcal{L}%
f\right\Vert _{H^{-1}}+\left\vert \left(  f,i\mathcal{R}\right)  \right\vert
+\left\vert \left(  f,\nabla\mathcal{R}\right)  \right\vert \right)  ,
\label{tw80}%
\end{equation}
for all $f\in H^{1}$.

Now, in order to prove (\ref{tw76}), let us compare $\mathcal{L}_{R_{N}}$ and
$\mathcal{L}$. Theorem \ref{T2} shows that there exist $x\left(  N\right)
,\gamma\left(  N\right)  \in\mathbb{R}$ such that $\tilde{R}_{N}%
:=e^{i\gamma\left(  N\right)  }R_{N}\left(  \cdot+x\left(  N\right)  \right)
\rightarrow\mathcal{R},$ in $H^{r},$ $r\geq0,$ as $N\rightarrow0.$ Since the
Sobolev norms are invariant under translations and phase-shift, to prove
(\ref{tw76}) we may assume that $R_{N}\ $itself converges to $\mathcal{R},$ in
$H^{r},$ $r\geq0,$ as $N\rightarrow0.$ We fix such $R_{N}\in\mathbf{R}_{N}$
and denote $\mathcal{L}_{N}:=\mathcal{L}_{R_{N}}$. Let $\eta_{N}\in L^{\infty
}$ be such that $\eta_{N}\left(  \xi\right)  =1,$ for $\left\vert
\xi\right\vert \leq\kappa^{-\alpha},$ $\kappa=N^{\frac{s}{2-s}},$
$0<\alpha<1/3,$ and $\eta_{N}\left(  \xi\right)  =0,$ for $\left\vert
\xi\right\vert \geq\kappa^{-\alpha}.$ We decompose $f\in H^{s/2}$ as
\begin{equation}
f=f_{1}+r,\label{tw125}%
\end{equation}
with%
\begin{equation}
f_{1}:=\mathcal{F}^{-1}\eta_{N}\mathcal{F}f\text{ \ and \ }r:=\mathcal{F}%
^{-1}\left(  1-\eta_{N}\right)  \mathcal{F}f.\label{tw82}%
\end{equation}
We have%
\begin{equation}
\left.
\begin{array}
[c]{c}%
\left\Vert \mathcal{L}_{N}f\right\Vert _{H^{-s/2}}^{2}=\left\Vert
\mathcal{L}_{N}f_{1}\right\Vert _{H^{-s/2}}^{2}+\left\Vert \mathcal{L}%
_{N}r\right\Vert _{H^{-s/2}}^{2}+2\left(  \left\langle D\right\rangle
^{-s}\mathcal{L}_{N}f_{1},\mathcal{L}_{N}r\right)  \\
\geq\left\Vert \mathcal{L}_{N}f_{1}\right\Vert _{H^{-s/2}}^{2}+\left\Vert
\mathcal{L}_{N}r\right\Vert _{H^{-s/2}}^{2}-2\left\vert \left(  \left\langle
D\right\rangle ^{-s}\mathcal{L}_{N}f_{1},\mathcal{L}_{N}r\right)  \right\vert
.
\end{array}
\right.  \label{tw124}%
\end{equation}
We denote%
\[
\mathcal{D=}\left(  \left\langle D\right\rangle ^{-s}n_{N}\left(  D\right)
\right)  ^{1/2}.
\]
Observe that%
\begin{equation}
\left.
\begin{array}
[c]{c}%
\left\vert \left(  \left\langle D\right\rangle ^{-s}\mathcal{L}_{N}%
f_{1},\mathcal{L}_{N}r\right)  \right\vert \leq\\
\leq%
{\displaystyle\int}
\left\vert \mathcal{D}f_{1}\right\vert \left\vert \mathcal{D}\left(  \left(
s+1\right)  \left\vert R_{N}\right\vert ^{2s}r+s\left(  \left\vert
R_{N}\right\vert ^{2s-2}R_{N}^{2}\right)  {\overline{r}}\right)  \right\vert
\\
+%
{\displaystyle\int}
\left\vert \mathcal{D}r\right\vert \left\vert \mathcal{D}\left(  \left(
s+1\right)  \left\vert R_{N}\right\vert ^{2s}f_{1}+s\left(  \left\vert
R_{N}\right\vert ^{2s-2}R_{N}^{2}\right)  {\overline{f}}_{1}\right)
\right\vert \\
+C%
{\displaystyle\int}
\left(  1+\left\vert \theta_{N}\right\vert +\left\vert R_{N}\right\vert
^{4s}\right)  \left\vert f_{1}\right\vert \left\vert r\right\vert
\end{array}
\right.  \label{tw121}%
\end{equation}
and%
\begin{equation}
\left.
\begin{array}
[c]{c}%
\left\Vert \mathcal{L}_{N}r\right\Vert _{H^{-s/2}}^{2}\geq\left\Vert
n_{N}\left(  D\right)  r\right\Vert _{H^{-s/2}}^{2}\\
-2%
{\displaystyle\int}
\left\vert \mathcal{D}r\right\vert \left\vert \mathcal{D}\left(  \theta
_{N}r-\left(  s+1\right)  \left\vert R_{N}\right\vert ^{2s}r-s\left(
\left\vert R_{N}\right\vert ^{2s-2}R_{N}^{2}\right)  {\overline{r}}\right)
\right\vert .
\end{array}
\right.  \label{tw122}%
\end{equation}
Also, we have%
\begin{equation}
\left\Vert \mathcal{L}_{N}f_{1}\right\Vert _{H^{-1}}^{2}\geq\frac{1}%
{2}\left\Vert \mathcal{L}f_{1}\right\Vert _{H^{-1}}^{2}-\left\Vert \left(
\mathcal{L}_{N}-\mathcal{L}\right)  f_{1}\right\Vert _{H^{-1}}^{2}%
.\label{tw123}%
\end{equation}
Using (\ref{tw121}), (\ref{tw122}) and (\ref{tw123}) in (\ref{tw124}) we get%
\begin{equation}
\left.  \left\Vert \mathcal{L}_{N}f\right\Vert _{H^{-s/2}}^{2}\geq\frac{1}%
{2}\left\Vert \mathcal{L}f_{1}\right\Vert _{H^{-1}}^{2}+\left\Vert
n_{N}\left(  D\right)  r\right\Vert _{H^{-s/2}}^{2}-\mathbf{r}_{1}\left(
f\right)  -\mathbf{r}_{2}\left(  f\right)  .\right.  \label{tw79}%
\end{equation}
where%
\[
\mathbf{r}_{1}\left(  f\right)  :=\left\Vert \left(  \mathcal{L}%
_{N}-\mathcal{L}\right)  f_{1}\right\Vert _{H^{-1}}^{2},
\]
and%
\[
\left.  \mathbf{r}_{2}\left(  f\right)  :=\mathbf{r}_{21}\left(  f\right)
+\mathbf{r}_{22}\left(  f\right)  +\mathbf{r}_{23}\left(  f\right)
+\mathbf{r}_{24}\left(  f\right)  ,\right.
\]
with%
\[
\mathbf{r}_{21}\left(  f\right)  :=2%
{\displaystyle\int}
\left\vert \mathcal{D}f_{1}\right\vert \left\vert \mathcal{D}\left(  \left(
s+1\right)  \left\vert R_{N}\right\vert ^{2s}r+s\left(  \left\vert
R_{N}\right\vert ^{2s-2}R_{N}^{2}\right)  {\overline{r}}\right)  \right\vert ,
\]%
\[
\mathbf{r}_{22}\left(  f\right)  :=2%
{\displaystyle\int}
\left\vert \mathcal{D}r\right\vert \left\vert \mathcal{D}\left(  \left(
s+1\right)  \left\vert R_{N}\right\vert ^{2s}f_{1}+s\left(  \left\vert
R_{N}\right\vert ^{2s-2}R_{N}^{2}\right)  {\overline{f}}_{1}\right)
\right\vert
\]%
\[
\left.  \mathbf{r}_{23}\left(  f\right)  :=2%
{\displaystyle\int}
\left\vert \mathcal{D}r\right\vert \left\vert \mathcal{D}\left(  \theta
_{N}r-\left(  s+1\right)  \left\vert R_{N}\right\vert ^{2s}r-s\left(
\left\vert R_{N}\right\vert ^{2s-2}R_{N}^{2}\right)  {\overline{r}}\right)
\right\vert .\right.
\]
and%
\[
\mathbf{r}_{24}\left(  f\right)  :=C%
{\displaystyle\int}
\left(  1+\left\vert \theta_{N}\right\vert +\left\vert R_{N}\right\vert
^{4s}\right)  \left\vert f_{1}\right\vert \left\vert r\right\vert ,
\]
It follows from (\ref{tw47}) that
\begin{equation}
\left\Vert n_{N}\left(  D\right)  r\right\Vert _{H^{-s/2}}^{2}\geq c\left\Vert
r\right\Vert _{H^{s/2}}^{2}.\label{tw126}%
\end{equation}
Moreover, from (\ref{tw80}) by (\ref{tw125}) we obtain%
\[
\left.  \frac{1}{2}\left\Vert \mathcal{L}f_{1}\right\Vert _{H^{-1}}^{2}\geq
c_{1}\left\Vert f_{1}\right\Vert _{H^{1}}^{2}-\left(  \left\vert \left(
f,i\mathcal{R}\right)  \right\vert ^{2}+\left\vert \left(  f,\nabla
\mathcal{R}\right)  \right\vert ^{2}+\left\vert \left(  r,i\mathcal{R}\right)
\right\vert ^{2}+\left\vert \left(  r,\nabla\mathcal{R}\right)  \right\vert
^{2}\right)  ,\right.
\]
for some $c_{1}>0.$ Then, using (\ref{tw126}) and $\left\Vert f_{1}\right\Vert
_{H^{1}}^{2}\geq c\left\Vert f_{1}\right\Vert _{H^{s/2}}^{2}$ ($s<2$), from
(\ref{tw79}) we deduce%
\begin{equation}
\left.
\begin{array}
[c]{c}%
\left\Vert \mathcal{L}_{N}f\right\Vert _{H^{-s/2}}^{2}\geq c_{2}\left\Vert
f\right\Vert _{H^{s/2}}^{2}-\left\vert \left(  f,i\mathcal{R}\right)
\right\vert ^{2}-\left\vert \left(  f,\nabla\mathcal{R}\right)  \right\vert
^{2}\\
-\mathbf{r}_{1}\left(  f\right)  -\mathbf{r}_{2}\left(  f\right)
-\mathbf{r}_{3}\left(  f\right)  ,
\end{array}
\right.  \label{tw81}%
\end{equation}
for some $c_{2}>0,$ where%
\[
\mathbf{r}_{3}\left(  f\right)  :=\left\vert \left(  r,i\mathcal{R}\right)
\right\vert ^{2}+\left\vert \left(  r,\nabla\mathcal{R}\right)  \right\vert
^{2}.
\]
Let us prove that there exists $N_{0}>0,$ such that for any $0<N<N_{0}$
\begin{equation}
\mathbf{r}_{1}\left(  f\right)  +\mathbf{r}_{2}\left(  f\right)
+\mathbf{r}_{3}\left(  f\right)  \leq\frac{c_{2}}{2}\left\Vert f\right\Vert
_{H^{s/2}}^{2},\label{tw90}%
\end{equation}
for all $f\in H^{1},$ with $c_{2}>0$ given by (\ref{tw81}). By the definition
(\ref{tw82}) of $r$ we have%
\begin{equation}
\left\Vert r\right\Vert _{H^{p}}^{2}=%
{\displaystyle\int\limits_{\mathbb{R}}}
\left\langle \xi\right\rangle ^{2p}\left\vert \hat{r}\left(  \xi\right)
\right\vert ^{2}d\xi\leq C\kappa^{\left(  s-2p\right)  \alpha}\left\Vert
f\right\Vert _{H^{s/2}}^{2},\text{ for }p<s/2.\label{tw83}%
\end{equation}
Then,
\begin{equation}
\mathbf{r}_{3}\left(  f\right)  \leq C\kappa^{s\alpha}\left\Vert f\right\Vert
_{H^{s/2}}^{2}.\label{tw91}%
\end{equation}
As $R_{N}$ converges to $\mathcal{R},$ in $H^{r},$ $r\geq0,$ as $N\rightarrow
0,$ there exists $N_{0}>0,$ such that
\begin{equation}
\left\Vert R_{N}\right\Vert _{H^{r}}\leq C,\label{tw96}%
\end{equation}
uniformly on $0<N<N_{0}.$ Then, via Sobolev theorem, using (\ref{tw86}) we
estimate $\mathbf{r}_{21}\left(  f\right)  $ as%
\[
\left.  \mathbf{r}_{21}\left(  f\right)  \leq C\left\Vert r\right\Vert
_{H^{1-s/2}}\left\Vert f_{1}\right\Vert _{H^{1-s/2}}.\right.
\]
Thus, taking into account (\ref{tw83}) with $p=1-s/2$, we see that%
\[
\mathbf{r}_{21}\left(  f\right)  \leq C\kappa^{\left(  s-1\right)  \alpha
}\left\Vert f\right\Vert _{H^{s/2}}\left\Vert f_{1}\right\Vert _{H^{1-s/2}%
}\leq C\kappa^{\left(  s-1\right)  \alpha}\left\Vert f\right\Vert _{H^{s/2}%
}^{2}.
\]
Similarly we estimate $\mathbf{r}_{22}\left(  f\right)  $ and $\mathbf{r}%
_{23}\left(  f\right)  $ (by using also (\ref{tw87})). Then, noting that by
(\ref{tw83}), (\ref{tw96}) and (\ref{tw87}), $\mathbf{r}_{24}\left(  f\right)
\leq C\kappa^{s\alpha/2}\left\Vert f\right\Vert _{H^{s/2}}^{2},$ we arrive to%
\begin{equation}
\mathbf{r}_{2}\left(  f\right)  \leq C\kappa^{\left(  s-1\right)  \alpha
}\left\Vert f\right\Vert _{H^{s/2}}^{2}.\label{tw92}%
\end{equation}
Let us consider $\mathbf{r}_{1}\left(  f\right)  .$ Observe that%
\[
\left.
\begin{array}
[c]{c}%
\left(  \mathcal{L}_{N}-\mathcal{L}\right)  f_{1}=\left(  n_{N}\left(
D\right)  -\left\vert D\right\vert ^{2}\right)  f_{1}+\left(  \theta
_{N}-\lambda\left(  s\right)  \right)  f_{1}\\
-\left(  s+1\right)  \left(  \left\vert R_{N}\right\vert ^{2s}-\mathcal{R}%
^{2s}\right)  f_{1}-s\left(  \left\vert R_{N}\right\vert ^{2s-2}R_{N}%
^{2}-\mathcal{R}^{2s}\right)  {\overline{f}}_{1}.
\end{array}
\right.
\]
Then, as $R_{N}$ converges to $\mathcal{R},$ by the definition (\ref{tw82}) of
$f_{1}$, (\ref{tw61}) and (\ref{tw102}), as $\kappa=N^{\frac{s}{2-s}}$ we
obtain
\begin{equation}
\mathbf{r}_{1}\left(  f\right)  =o\left(  1\right)  \left\Vert f\right\Vert
_{H^{s/2}}^{2},\text{ as }N\rightarrow0.\label{tw93}%
\end{equation}
Using (\ref{tw91}), (\ref{tw92}) and (\ref{tw93}), we attain (\ref{tw90}%
).\ Introducing (\ref{tw90}) into (\ref{tw81}) we get%
\begin{equation}
\left.  \left\Vert f\right\Vert _{H^{s/2}}\leq C\left(  \left\Vert
\mathcal{L}_{N}f\right\Vert _{H^{-s/2}}+\left\vert \left(  f,i\mathcal{R}%
\right)  \right\vert +\left\vert \left(  f,\nabla\mathcal{R}\right)
\right\vert \right)  .\right.  \label{tw94}%
\end{equation}
Since $R_{N}$ converges to $\mathcal{R},$
\[
\left\vert \left(  f,i\left(  \mathcal{R}-R_{N}\right)  \right)  \right\vert
+\left\vert \left(  f,\nabla\left(  \mathcal{R}-R_{N}\right)  \right)
\right\vert \leq o\left(  1\right)  \left\Vert f\right\Vert _{H^{s/2}},
\]
as $N\rightarrow0.$ Hence, $\left\vert \left(  f,i\mathcal{R}\right)
\right\vert +\left\vert \left(  f,\nabla\mathcal{R}\right)  \right\vert $ in
(\ref{tw94}) may be replaced by $\left\vert \left(  f,iR_{N}\right)
\right\vert +\left\vert \left(  f,\nabla R_{N}\right)  \right\vert .$
Therefore, we arrive to (\ref{tw76}). Now, note that (\ref{tw76}) implies
$\ker\mathcal{L}_{N}=\operatorname*{span}\left\{  iR_{N},\nabla R_{N}\right\}
$. Then, the second statement of Lemma \ref{L8} follows from Fredholm
alternative applied to the operator $\left(  n_{N}\left(  D\right)
+\theta_{N}\right)  ^{-1}\mathcal{L}_{R_{N}}$.
\end{proof}

\begin{proof}
[Proof of Lemma \ref{L7}]Without loss of generality we suppose that $R_{N}$
and $\tilde{R}_{N}$ tend to $\mathcal{R},$ as $N\rightarrow0.$ For any
$\gamma,y\in\mathbb{R}$ we define%
\[
\varepsilon_{N}\left(  x,\gamma,y\right)  :=\tilde{R}_{N}\left(  x\right)
-e^{i\gamma}R_{N}\left(  x-y\right)
\]
and%
\[
f_{N}\left(  \gamma,y\right)  :=\left(  \varepsilon_{N}\left(  x,\gamma
,y\right)  ,i\tilde{R}_{N}\right)  ,\text{ }g_{N}\left(  \gamma,y\right)
:=\left(  \varepsilon_{N}\left(  x,\gamma,y\right)  ,\nabla\tilde{R}%
_{N}\right)  .
\]
These functions are smooth with respect to $\gamma$ and $y.$ Let us denote by
$\mathcal{J}_{N}\left(  \gamma,y\right)  $ the Jacobian matrix of $f_{N}$ and
$g_{N}$ at $\left(  \gamma,y\right)  .$ As $R_{N}$ converges to $\mathcal{R},$
in $H^{r},$ $r\geq0,$ as $N\rightarrow0,$ there is $N_{0}>0,$ such that for
all $0<N<$ $N_{0},$%
\[
\left\vert \det\mathcal{J}_{N}\left(  \gamma,y\right)  \right\vert \geq
\frac{1}{2}\left\Vert \mathcal{R}\right\Vert _{L^{2}}^{2}\left\Vert
\nabla\mathcal{R}\right\Vert _{L^{2}}^{2},
\]
for $\left(  \gamma,y\right)  $ in some neighborhood of $\left(  0,0\right)  $
(dependent only on $N_{0}$). Then, as $f_{N}\left(  0,0\right)  $ and
$g_{N}\left(  0,0\right)  $ tend to $0$, as $N\rightarrow0,$ we conclude that
there are $\gamma\left(  N\right)  ,y\left(  N\right)  \in\mathbb{R}$ such
that
\begin{equation}
f_{N}\left(  \gamma\left(  N\right)  ,y\left(  N\right)  \right)
=g_{N}\left(  \gamma\left(  N\right)  ,y\left(  N\right)  \right)  =0.
\label{tw95}%
\end{equation}
Since $R_{N}$ and $\tilde{R}_{N}$ satisfy equation (\ref{tw60}) and $R_{N}$
converges to $\mathcal{R}$, we have%
\[
\left\Vert \mathcal{L}_{\tilde{R}_{N}}\varepsilon_{N}\left(  \cdot
,\gamma\left(  N\right)  ,y\left(  N\right)  \right)  \right\Vert _{H^{-s/2}%
}\leq o\left(  1\right)  \left\Vert \varepsilon_{N}\left(  \cdot,\gamma\left(
N\right)  ,y\left(  N\right)  \right)  \right\Vert _{H^{s/2}}.
\]
Therefore, using (\ref{tw97}) we conclude that $\varepsilon_{N}\left(
x,\gamma\left(  N\right)  ,y\left(  N\right)  \right)  \equiv0.$ Lemma
\ref{L7} is proved.
\end{proof}

\section{Spatial asymptotics
of travelling waves.}

This Section is devoted to the proof of Theorem \ref{Th4}.\ For this purpose
we use relations (\ref{tw119}) and (\ref{tw116}) and consider the function
$R_{N}$ which satisfies the equation (\ref{tw115}). Taking the Fourier
transform in the both sides of (\ref{tw115}) we have
\[
n_{N}\left(  \xi\right)  \hat{R}_{N}\left(  \xi\right)  +\theta_{N}\hat{R}%
_{N}\left(  \xi\right)  =\left(  \mathcal{F}\left(  \left\vert R_{N}%
\right\vert ^{2s}R_{N}\right)  \right)  \left(  \xi\right)  ,
\]
where we recall that%
\[
n_{N}\left(  \xi\right)  =\frac{2n\left(  N^{\frac{s}{2-s}}\xi\right)
}{s\left(  s-1\right)  N^{\frac{2s}{2-s}}}=\frac{2\left(  \left\vert
N^{\frac{s}{2-s}}\xi+1\right\vert ^{s}-sN^{\frac{s}{2-s}}\xi-1\right)
}{s\left(  s-1\right)  N^{\frac{2s}{2-s}}}.
\]
(Here we used formulae (\ref{tw127}) and (\ref{tw114})). Then, \
\begin{equation}
R_{N}=m_{N}\ast\left(  \left\vert R_{N}\right\vert ^{2s}R_{N}\right)
\label{tw148}%
\end{equation}
where
\[
m_{N}\left(  x\right)  :=\left(  \mathcal{F}^{-1}\left(  \frac{1}{n_{N}\left(
\cdot\right)  +\theta_{N}}\right)  \right)  \left(  x\right)  .
\]
We begin by studying the asymptotics of the function $m_{N}\left(  x\right)
.$ Recall that%
\[
\mathcal{C}_{1}=\sqrt{\frac{\pi}{2\lambda\left(  s\right)  }}\text{ \ and
}\mathcal{C}_{2}=\dfrac{\left(  si^{s+1}+\left(  -i\right)  ^{s+1}\right)
e^{-i\frac{x}{\kappa}}}{2\sqrt{2\pi}\left(  s-1\right)  }\Gamma\left(
s\right)  .
\]
($\Gamma\left(  s\right)  $ denotes the Gamma function.) We prove the following.

\begin{lemma}
\label{L10}The following expansion is true%
\begin{equation}
\left.
\begin{array}
[c]{c}%
m_{N}\left(  x\right)  =\mathcal{C}_{1}e^{-\sqrt{\lambda\left(  s\right)
}\left\vert x\right\vert }+\mathcal{C}_{2}\frac{N^{\frac{s\left(  2+s\right)
}{2-s}}}{\left\vert x\right\vert ^{s+1}}+o_{N}\left(  1\right)  e^{-\sqrt
{\lambda\left(  s\right)  }\left\vert x\right\vert }\\
+\left(  o_{N}\left(  1\right)  +o_{\left\vert x\right\vert }\left(  1\right)
\right)  \frac{N^{\frac{s\left(  2+s\right)  }{2-s}}}{\left\vert x\right\vert
^{s+1}},
\end{array}
\right.  \label{as45}%
\end{equation}
for $\left\vert x\right\vert \rightarrow\infty$ and $N\rightarrow0,$ where
$o_{N}\left(  1\right)  \rightarrow0,$ as $N\rightarrow0,$ and $o_{\left\vert
x\right\vert }\left(  1\right)  \rightarrow0,$ as $\left\vert x\right\vert
\rightarrow\infty.$
\end{lemma}

Before proving Lemma \ref{L10}, we prepare a result that is involved in its
proof. Let us consider the functions%
\begin{equation}
f_{1}^{\pm}\left(  y\right)  =\left(  y+1\right)  ^{s}-sy-1+\frac{s\left(
s-1\right)  }{2}N^{\frac{2s}{2-s}}\theta_{N} \label{tw146}%
\end{equation}
on $G_{1}^{\pm}=\left\{  y\in\mathbb{C}:\operatorname{Re}y>-1\text{ and }%
\pm\operatorname{Im}y>0\right\}  ,$ with the brunch of $\left(  y+1\right)
^{s}$ selected in such way that $\left(  x+1\right)  ^{s}=\left(  x+1\right)
^{s}$, for $x>-1.$ Also, let
\begin{equation}
f_{2}^{\pm}\left(  y\right)  =\left(  y-1\right)  ^{s}+sy-1+N^{\frac{2s}{2-s}%
}\theta_{N} \label{tw147}%
\end{equation}
on $G_{2}^{\pm}=\left\{  y\in\mathbb{C}:\operatorname{Re}y>1\text{ and }%
\pm\operatorname{Im}y>0\right\}  ,$ where $\left(  x-1\right)  ^{s}=\left(
x-1\right)  ^{s},$ for $x>1.$ We have the following.

\begin{lemma}
\label{L9}i) There is $N_{0}>0$, such that for any $0\leq N\leq N_{0},$ the
function $f_{1}^{\pm}\left(  y\right)  $ has only one root $y^{\pm}=y^{\pm
}\left(  N\right)  $ in the region $G_{1}^{\pm}.$ This root satisfies the
estimate%
\begin{equation}
\left\vert y^{\pm}\left(  N\right)  \right\vert =O\left(  N^{\frac{s}{2-s}%
}\right)  , \label{tw133}%
\end{equation}
as $N$ tends to $0.$

ii) On the other hand, the function $f_{2}^{\pm}\left(  y\right)  $ has no
roots in the region $G_{2}^{\pm}.$
\end{lemma}

\begin{proof}
Let us consider the case of $f_{1}^{+}.$ We translate $y\rightarrow y-1$ and
study the zeros of the function $\tilde{f}_{1}^{+}\left(  y\right)
=y^{s}-sy+s-1+\kappa^{2}\theta^{\left(  \kappa\right)  }$ in $\tilde{G}%
_{1}^{+}=\{y\in\mathbb{C}:\operatorname{Re}y\geq0$ and $\operatorname{Im}%
y\geq0\}$. We write $y=\left\vert y\right\vert e^{i\phi},$ for $0\leq\phi
\leq\frac{\pi}{2}.$ Then, we need to solve the following equation
\[
\left.  \tilde{f}_{1}^{+}\left(  y\right)  =\left\vert y\right\vert
^{s}e^{is\phi}-s\left\vert y\right\vert e^{i\phi}+s-1+N^{\frac{2s}{2-s}}%
\theta_{N}=f_{11}^{+}\left(  \left\vert y\right\vert ,\phi\right)
+if_{12}^{+}\left(  \left\vert y\right\vert ,\phi\right)  =0,\right.
\]
with%
\[
f_{11}^{+}\left(  \left\vert y\right\vert ,\phi\right)  =\left\vert
y\right\vert ^{s}\cos\left(  s\phi\right)  -s\left\vert y\right\vert \cos
\phi+s-1+N^{\frac{2s}{2-s}}\theta_{N}%
\]
and%
\[
f_{12}^{+}\left(  \left\vert y\right\vert ,\phi\right)  =\left\vert
y\right\vert ^{s}\sin\left(  s\phi\right)  -s\left\vert y\right\vert \sin
\phi.
\]
Equivalently, we get the equations
\begin{equation}
f_{11}^{+}\left(  \left\vert y\right\vert ,\phi\right)  =0 \label{tw128}%
\end{equation}
and
\begin{equation}
f_{12}^{+}\left(  \left\vert y\right\vert ,\phi\right)  =0. \label{tw129}%
\end{equation}
From (\ref{tw129}) we see that for $0<\phi\leq\frac{\pi}{2}$%
\begin{equation}
\left\vert y\right\vert ^{s-1}=s\frac{\sin\phi}{\sin\left(  s\phi\right)  }.
\label{tw132}%
\end{equation}
The right-hand side is increasing on $0\leq\phi\leq\frac{\pi}{2}$ and it is
equal to $1$ for $\phi=0.$ Therefore, there is no roots for $f_{1}^{+}\left(
y\right)  $ if $\left\vert y\right\vert <1.$ Let $r\left(  \phi\right)
=\left(  s\frac{\sin\phi}{\sin\left(  s\phi\right)  }\right)  ^{\frac{1}{s-1}%
}.$ Then, we need to solve $f_{11}^{+}\left(  r\left(  \phi\right)
,\phi\right)  =0$ on $0\leq\phi\leq\frac{\pi}{2}.$ Note that for $0<\phi
\leq\frac{\pi}{2}$%
\begin{equation}
\left.
\begin{array}
[c]{c}%
\frac{d}{d\phi}f_{11}^{+}\left(  r\left(  \phi\right)  ,\phi\right)  =s\left(
r^{s-1}\left(  \phi\right)  \cos\left(  s\phi\right)  -\cos\phi\right)
\frac{dr\left(  \phi\right)  }{d\phi}\\
-sr\left(  r^{s-1}\left(  \phi\right)  \sin\left(  s\phi\right)  -\sin
\phi\right) \\
=\frac{s}{\sin\left(  s\phi\right)  }\left(  s\sin\phi\cos\left(
s\phi\right)  -\sin\left(  s\phi\right)  \cos\phi\right)  \frac{dr\left(
\phi\right)  }{d\phi}\\
-s\left(  s-1\right)  r\left(  \phi\right)  \sin\phi<0.
\end{array}
\right.  \label{tw130}%
\end{equation}
Moreover, $f_{11}^{+}\left(  r\left(  0\right)  ,0\right)  =N^{\frac{2s}{2-s}%
}\theta_{N}.$ Then, using (\ref{tw102}), we see that $f_{11}^{+}\left(
r\left(  0\right)  ,0\right)  >0$ for $N>0$ small enough. Since $r\left(
\phi\right)  $ is increasing on $0\leq\phi\leq\frac{\pi}{2},$ from
(\ref{tw130}) we get%
\[
\left.
\begin{array}
[c]{c}%
\frac{d}{d\phi}f_{11}^{+}\left(  r\left(  \phi\right)  ,\phi\right) \\
=\frac{s}{\sin\left(  s\phi\right)  }\left(  s\sin\phi\cos\left(
s\phi\right)  -\sin\left(  s\phi\right)  \cos\phi\right)  \frac{dr\left(
\phi\right)  }{d\phi}\\
-s\left(  s-1\right)  \left(  r\left(  \phi\right)  -1\right)  \sin
\phi-s\left(  s-1\right)  \sin\phi<-s\left(  s-1\right)  \sin\phi.
\end{array}
\right.
\]
Integrating the last inequality we see that%
\[
f_{11}^{+}\left(  r\left(  \phi\right)  ,\phi\right)  <N^{\frac{2s}{2-s}%
}\theta_{N}-s\left(  s-1\right)  \left(  1-\cos\phi\right)  .
\]
Thus, there is $0<\phi_{0}<\frac{\pi}{2},$ such that $f_{11}^{+}\left(
r\left(  \phi_{0}\right)  ,\phi_{0}\right)  <0,$ for all $N>0$ small enough.
Hence, we conclude that there is $N_{0}>0$, such that for any $0\leq N\leq
N_{0},$ there is $0<\phi\left(  N\right)  <\frac{\pi}{2}$ with the property
\[
f_{11}^{+}\left(  r\left(  \phi\left(  N\right)  \right)  ,\phi\left(
N\right)  \right)  =0.
\]
Therefore we conclude for $N>0$ sufficiently small, the function $f_{1}%
^{+}\left(  y\right)  $ has only one root $y^{+}=y^{+}\left(  N\right)  .$
Since $\left(  y+1\right)  ^{s}-sy-1=0$ only for $y=0,$ we show that $y\left(
N\right)  \rightarrow0,$ as $N$ tends to $0.$ Then, using that $\left(
y+1\right)  ^{s}-sy-1=Cy^{2}+o\left(  y^{2}\right)  ,$ as $y\rightarrow0$,
from the equation $f_{1}^{+}\left(  y\left(  N\right)  \right)  =0$ we get
(\ref{tw133}). To prove the same result for $f_{1}^{-},$ we again translate
$y\rightarrow y-1$ and study the zeros of the function $\tilde{f}_{1}%
^{-}\left(  y\right)  =y^{s}-sy+s-1+\kappa^{2}\theta^{\left(  \kappa\right)
}$ in $\tilde{G}_{1}^{-}=\{y\in\mathbb{C}:\operatorname{Re}y\geq0$ and
$\operatorname{Im}y\leq0\}$. We represent $y=\left\vert y\right\vert
e^{-i\phi},$ for $0\leq\phi\leq\frac{\pi}{2}.$ Then,%
\[
\left.  \tilde{f}_{1}^{-}\left(  y\right)  =f_{11}^{-}\left(  \left\vert
y\right\vert ,\phi\right)  +if_{12}^{-}\left(  \left\vert y\right\vert
,\phi\right)  =0,\right.
\]
with%
\[
f_{11}^{-}\left(  \left\vert y\right\vert ,\phi\right)  =\left\vert
y\right\vert ^{s}\cos\left(  s\phi\right)  -s\left\vert y\right\vert \cos
\phi+s-1+N^{\frac{2s}{2-s}}\theta_{N}%
\]
and%
\[
f_{12}^{-}\left(  \left\vert y\right\vert ,\phi\right)  =-\left\vert
y\right\vert ^{s}\sin\left(  s\phi\right)  +s\left\vert y\right\vert \sin
\phi.
\]
Arguing similarly to the case of $f_{1}^{+}$ we prove that $f_{1}^{-}$ has
only one root $y^{-}=y^{-}\left(  N\right)  $ which satisfies (\ref{tw133}).
This proves the first part of Lemma \ref{L9}.

To prove the second part, we first translate $y\rightarrow y+1$ and study the
existence of zeros for
\begin{equation}
\tilde{f}_{2}^{\pm}\left(  y\right)  =y^{s}+sy+s-1+N^{\frac{2s}{2-s}}%
\theta_{N} \label{tw131}%
\end{equation}
on $\tilde{G}_{2}^{\pm}=\left\{  y\in\mathbb{C}:\operatorname{Re}y\geq0\text{
and }\pm\operatorname{Im}y\geq0\right\}  .$ We introduce the decomposition
$y=\left\vert y\right\vert e^{\pm i\phi},$ for $0\leq\phi\leq\frac{\pi}{2}$
into (\ref{tw131}) to get%
\[
\tilde{f}_{2}^{\pm}\left(  y\right)  =f_{21}^{\pm}\left(  \left\vert
y\right\vert ,\phi\right)  +if_{22}^{\pm}\left(  \left\vert y\right\vert
,\phi\right)  ,
\]
with%
\[
f_{21}^{\pm}\left(  \left\vert y\right\vert ,\phi\right)  =\left\vert
y\right\vert ^{s}\cos\left(  s\phi\right)  +s\left\vert y\right\vert \cos
\phi+s-1+N^{\frac{2s}{2-s}}\theta_{N}%
\]
and%
\[
f_{22}^{\pm}\left(  \left\vert y\right\vert ,\phi\right)  =\pm\left(
\left\vert y\right\vert ^{s}\sin\left(  s\phi\right)  +s\left\vert
y\right\vert \sin\phi\right)  .
\]
Since
\[
\pm f_{22}^{\pm}\left(  \left\vert y\right\vert ,\phi\right)  >0
\]
for all $\left\vert y\right\vert >0$ and all $0<\phi\leq\frac{\pi}{2},$ and%
\[
f_{21}^{\pm}\left(  \left\vert y\right\vert ,0\right)  =\left\vert
y\right\vert ^{s}+s\left\vert y\right\vert +s-1+N^{\frac{2s}{2-s}}\theta
_{N}>0,
\]
we conclude that $\tilde{f}_{2}^{\pm}\left(  y\right)  $ does not have roots
on $\tilde{G}_{2}^{\pm}$.
\end{proof}

\begin{proof}
[Proof of Lemma \ref{L10}]We have \
\[
m^{\left(  \kappa\right)  }\left(  x\right)  :=m_{N}\left(  x\right)
=\frac{s\left(  s-1\right)  }{2}\mathcal{F}^{-1}\left(  \frac{1}{\kappa
^{-2}\left(  \left\vert \kappa\xi+1\right\vert ^{s}-s\kappa\xi-1\right)
+\theta^{\left(  \kappa\right)  }}\right)  \left(  x\right)  ,
\]
with $\kappa=N^{\frac{s}{2-s}}$ and $\theta^{\left(  \kappa\right)  }%
=\frac{s\left(  s-1\right)  }{2}\theta_{N}.$ Let us study the function
\[
I:=\mathcal{F}^{-1}\left(  \frac{1}{\kappa^{-2}\left(  \left\vert \kappa
\xi+1\right\vert ^{s}-s\kappa\xi-1\right)  +\theta^{\left(  \kappa\right)  }%
}\right)  .
\]
We have
\begin{equation}
\left.
\begin{array}
[c]{c}%
I=\dfrac{1}{\sqrt{2\pi}}%
{\displaystyle\int\limits_{-\infty}^{\infty}}
e^{ix\xi}\dfrac{d\xi}{\kappa^{-2}\left(  \left\vert \kappa\xi+1\right\vert
^{s}-s\kappa\xi-1\right)  +\theta^{\left(  \kappa\right)  }}\\
=\dfrac{1}{\sqrt{2\pi}}%
{\displaystyle\int\limits_{-\frac{1}{\kappa}}^{\infty}}
e^{ix\xi}\dfrac{d\xi}{\kappa^{-2}\left(  \left\vert \kappa\xi+1\right\vert
^{s}-s\kappa\xi-1\right)  +\theta^{\left(  \kappa\right)  }}\\
+\dfrac{1}{\sqrt{2\pi}}%
{\displaystyle\int\limits_{-\infty}^{-\frac{1}{\kappa}}}
e^{ix\xi}\dfrac{d\xi}{\kappa^{-2}\left(  \left\vert \kappa\xi+1\right\vert
^{s}-s\kappa\xi-1\right)  +\theta^{\left(  \kappa\right)  }}\\
=I_{1}+I_{2},
\end{array}
\right.  \label{tw145}%
\end{equation}
with
\[
I_{1}=\dfrac{1}{\sqrt{2\pi}}%
{\displaystyle\int\limits_{-\frac{1}{\kappa}}^{\infty}}
e^{ix\xi}\dfrac{d\xi}{\kappa^{-2}\left(  \left(  \kappa\xi+1\right)
^{s}-s\kappa\xi-1\right)  +\theta^{\left(  \kappa\right)  }}%
\]
and%
\[
I_{2}=\dfrac{1}{\sqrt{2\pi}}%
{\displaystyle\int\limits_{\frac{1}{\kappa}}^{\infty}}
e^{-ix\xi}\dfrac{d\xi}{\kappa^{-2}\left(  \left(  \kappa\xi-1\right)
^{s}+s\kappa\xi-1\right)  +\theta^{\left(  \kappa\right)  }}.
\]

Suppose that $x>0.$ First, we consider $I_{1}.$ We extend the denominator
$F\left(  \xi\right)  =\kappa^{-2}\left(  \left(  \kappa\xi+1\right)
^{s}-s\kappa\xi-1\right)  +\theta^{\left(  \kappa\right)  }$ analytically by
the function $\kappa^{-2}f_{1}^{+}\left(  \kappa\xi\right)  ,$ defined by
(\ref{tw146}). By Lemma \ref{L9} $F\left(  \xi\right)  =\kappa^{-2}\left(
\left(  \kappa\xi+1\right)  ^{s}-s\kappa\xi-1\right)  +\theta^{\left(
\kappa\right)  }$ has only one root $\xi=\xi^{\left(  \kappa\right)  }$ in the
region $\left\{  \xi\in\mathbb{C}:\operatorname{Re}\xi>-\frac{1}{\kappa}\text{
and }\operatorname{Im}\xi>0\right\}  .$ Then, it follows from Jordan's lemma
that%
\begin{equation}
\left.  I_{1}=\dfrac{2\pi i}{\sqrt{2\pi}}e^{ix\xi^{\left(  \kappa\right)  }%
}\dfrac{\kappa}{s\left(  \kappa\xi^{\left(  \kappa\right)  }+1\right)
^{s-1}-s}+I_{11},\right.  \label{tw141}%
\end{equation}
where
\[
I_{11}=\dfrac{1}{\sqrt{2\pi}}%
{\displaystyle\int\limits_{-\frac{1}{\kappa}}^{-\frac{1}{\kappa}+i\infty}}
e^{ix\xi}\dfrac{d\xi}{\kappa^{-2}\left(  \left(  \kappa\xi+1\right)
^{s}-s\kappa\xi-1\right)  +\theta^{\left(  \kappa\right)  }}.
\]
By (\ref{tw134})%
\begin{equation}
F\left(  \xi^{\left(  \kappa\right)  }\right)  =\frac{s\left(  s-1\right)
}{2}\left(  \xi^{\left(  \kappa\right)  }\right)  ^{2}+\theta^{\left(
\kappa\right)  }+O\left(  \kappa\left(  \xi^{\left(  \kappa\right)  }\right)
^{3}\right)  =0. \label{tw135}%
\end{equation}
Using (\ref{tw133}) we deduce that $\left\vert \xi^{\left(  \kappa\right)
}\right\vert \leq C.$ Then, from (\ref{tw135}) it follows%
\[
\frac{s\left(  s-1\right)  }{2}\left(  \xi^{\left(  \kappa\right)  }\right)
^{2}+\theta^{\left(  \kappa\right)  }=O\left(  \kappa\right)  .
\]
As $\operatorname{Im}\xi^{\left(  \kappa\right)  }\geq0,$ using $\theta
^{\left(  \kappa\right)  }=\frac{s\left(  s-1\right)  }{2}\theta_{N}$ and
(\ref{tw102}) we get%
\[
\xi^{\left(  \kappa\right)  }-\sqrt{\lambda\left(  s\right)  }i=o\left(
1\right)  ,
\]
as $\kappa$ $\rightarrow0.$ Therefore, taking into account the relation
\[
s\left(  \kappa\xi^{\left(  \kappa\right)  }+1\right)  ^{s-1}-s=s\left(
s-1\right)  \sqrt{\lambda\left(  s\right)  }i\kappa+o\left(  \kappa\right)
\]
we get
\begin{equation}
\left.
\begin{array}
[c]{c}%
\dfrac{2\pi i}{\sqrt{2\pi}}e^{ix\xi^{\left(  \kappa\right)  }}\dfrac{\kappa
}{s\left(  \kappa\xi^{\left(  \kappa\right)  }+1\right)  ^{s-1}-s}\\
=\sqrt{\frac{\pi}{2}}\dfrac{2}{s\left(  s-1\right)  \sqrt{\lambda\left(
s\right)  }}e^{-\sqrt{\lambda\left(  s\right)  }\left\vert x\right\vert
}+e^{-\sqrt{\lambda\left(  s\right)  }\left\vert x\right\vert }o\left(
1\right)  ,
\end{array}
\right.  \label{tw140}%
\end{equation}
as $\kappa$ $\rightarrow0.$

Making the change $y=-i\left(  \kappa\xi+1\right)  $ in the integral in
$I_{11}$ we have%
\[
I_{11}=\dfrac{i}{\sqrt{2\pi}}e^{-i\frac{x}{\kappa}}%
{\displaystyle\int\limits_{0}^{\infty}}
e^{-\frac{x}{\kappa}y}\dfrac{\kappa dy}{\left(  iy\right)  ^{s}-isy+s-1+\kappa
^{2}\theta^{\left(  \kappa\right)  }}.
\]
Integrating by parts in $I_{1}$ we have%
\begin{equation}
I_{11}=\dfrac{i\kappa^{2}}{x\sqrt{2\pi}}\dfrac{e^{-i\frac{x}{\kappa}}%
}{s-1+\kappa^{2}\theta^{\left(  \kappa\right)  }}+Z, \label{tw136}%
\end{equation}
with%
\[
Z=\dfrac{i\kappa^{2}e^{-i\frac{x}{\kappa}}}{x\sqrt{2\pi}}%
{\displaystyle\int\limits_{0}^{\infty}}
e^{-\frac{x}{\kappa}y}\dfrac{\left(  i^{s}sy^{s-1}-is\right)  }{\left(
i^{s}y^{s}-isy+s-1+\kappa^{2}\theta^{\left(  \kappa\right)  }\right)  ^{2}%
}dy.
\]
We decompose now $Z$ as%
\begin{equation}
Z=Z_{1}+Z_{2}+Z_{3}, \label{tw139}%
\end{equation}
where
\[
Z_{1}=i^{s+1}\dfrac{s\kappa^{2}e^{-i\frac{x}{\kappa}}}{\left(  s-1\right)
^{2}x\sqrt{2\pi}}%
{\displaystyle\int\limits_{0}^{\infty}}
e^{-\frac{x}{\kappa}y}y^{s-1}dy,
\]%
\[
Z_{2}=\dfrac{s\kappa^{2}e^{-i\frac{x}{\kappa}}}{x\sqrt{2\pi}}%
{\displaystyle\int\limits_{0}^{\infty}}
e^{-\frac{x}{\kappa}y}\dfrac{1}{\left(  i^{s}y^{s}-isy+s-1+\kappa^{2}%
\theta^{\left(  \kappa\right)  }\right)  ^{2}}dy
\]
and%
\[
Z_{3}=\dfrac{si^{s+1}\kappa^{2}e^{-i\frac{x}{\kappa}}}{x\sqrt{2\pi}}%
{\displaystyle\int\limits_{0}^{\infty}}
e^{-\frac{x}{\kappa}y}y^{s-1}\left(  \dfrac{1}{\left(  i^{s}y^{s}%
-isy+s-1+\kappa^{2}\theta^{\left(  \kappa\right)  }\right)  ^{2}}-\frac
{1}{\left(  s-1\right)  ^{2}}\right)  dy.
\]
Making the change $z=\frac{x}{\kappa}y$ in $Z_{1}$ and $Z_{3}$ we have%
\begin{equation}
Z_{1}=\dfrac{si^{s+1}\kappa^{2+s}e^{-i\frac{x}{\kappa}}}{\left(  s-1\right)
^{2}x^{s+1}\sqrt{2\pi}}%
{\displaystyle\int\limits_{0}^{\infty}}
e^{-z}z^{s-1}dz=\dfrac{si^{s+1}e^{-i\frac{x}{\kappa}}}{\left(  s-1\right)
^{2}\sqrt{2\pi}}\Gamma\left(  s\right)  \left(  \frac{\kappa^{2+s}}{x^{s+1}%
}\right)  \label{tw137}%
\end{equation}
and%
\begin{equation}
Z_{3}=\frac{o\left(  \kappa^{2+s}\right)  }{x^{s+1}}+\kappa^{2+s}o\left(
\frac{1}{x^{s+1}}\right)  . \label{tw138}%
\end{equation}
Using (\ref{tw137}) and (\ref{tw138}) in (\ref{tw139}) we have%
\[
Z=\dfrac{si^{s+1}e^{-i\frac{x}{\kappa}}}{\left(  s-1\right)  ^{2}\sqrt{2\pi}%
}\Gamma\left(  s\right)  \left(  \frac{\kappa^{2+s}}{x^{s+1}}\right)
+Z_{2}+\frac{o\left(  \kappa^{2+s}\right)  }{x^{s+1}}+\kappa^{2+s}o\left(
\frac{1}{x^{s+1}}\right)  .
\]
Introducing the last equation into (\ref{tw136}) we get
\begin{equation}
\left.
\begin{array}
[c]{c}%
I_{11}=\dfrac{i\kappa^{2}}{x\sqrt{2\pi}}\dfrac{e^{-i\frac{x}{\kappa}}%
}{s-1+\kappa^{2}\theta^{\left(  \kappa\right)  }}+Z_{2}\\
+\dfrac{si^{s+1}e^{-i\frac{x}{\kappa}}}{\left(  s-1\right)  ^{2}\sqrt{2\pi}%
}\Gamma\left(  s\right)  \left(  \frac{\kappa^{2+s}}{x^{s+1}}\right)
+\frac{o\left(  \kappa^{2+s}\right)  }{x^{s+1}}+\kappa^{2+s}o\left(  \frac
{1}{x^{s+1}}\right)  .
\end{array}
\right.  \label{tw143}%
\end{equation}
Finally, using (\ref{tw140}) and (\ref{tw143}) in (\ref{tw141})%
\begin{equation}
\left.
\begin{array}
[c]{c}%
I_{1}=\sqrt{\frac{\pi}{2}}\dfrac{2}{s\left(  s-1\right)  \sqrt{\lambda\left(
s\right)  }}e^{-\sqrt{\lambda\left(  s\right)  }\left\vert x\right\vert }\\
+\dfrac{si^{s+1}e^{-i\frac{x}{\kappa}}}{\left(  s-1\right)  ^{2}\sqrt{2\pi}%
}\Gamma\left(  s\right)  \left(  \frac{\kappa^{2+s}}{x^{s+1}}\right)
+\dfrac{i\kappa^{2}}{x\sqrt{2\pi}}\dfrac{e^{-i\frac{x}{\kappa}}}%
{s-1+\kappa^{2}\theta^{\left(  \kappa\right)  }}\\
+Z_{2}+e^{-\sqrt{\lambda\left(  s\right)  }\left\vert x\right\vert }o_{\kappa
}\left(  1\right)  +\frac{\kappa^{2+s}}{x^{s+1}}o_{\kappa}\left(  1\right)
+\frac{\kappa^{2+s}}{x^{s+1}}o_{\left\vert x\right\vert }\left(  1\right)  ,
\end{array}
\right.  \label{tw142}%
\end{equation}
where $o_{\kappa}\left(  1\right)  \rightarrow0,$ as $\kappa\rightarrow0,$ and
$o_{\left\vert x\right\vert }\left(  1\right)  \rightarrow0,$ as $\left\vert
x\right\vert \rightarrow\infty.$

Let us now consider $I_{2}.$ We extend $F_{1}\left(  \xi\right)  =\kappa
^{-2}\left(  \left(  \kappa\xi-1\right)  ^{s}+s\kappa\xi-1\right)
+\theta^{\left(  \kappa\right)  }$ to the analytic function $\kappa^{-2}%
f_{2}^{-}\left(  \kappa\xi\right)  $, defined by (\ref{tw147}). By Lemma
\ref{L9} the denominator $F_{1}\left(  \xi\right)  =\kappa^{-2}\left(  \left(
\kappa\xi-1\right)  ^{s}+s\kappa\xi-1\right)  +\theta^{\left(  \kappa\right)
}$ has no roots in the region $\left\{  \xi\in\mathbb{C}:\operatorname{Re}%
\xi>\frac{1}{\kappa}\text{ and }\operatorname{Im}\xi<0\right\}  .$ Then, by
Jordan's lemma we have
\[
I_{2}=-\dfrac{1}{\sqrt{2\pi}}%
{\displaystyle\int\limits_{\frac{1}{\kappa}-i\infty}^{\frac{1}{\kappa}}}
e^{-ix\xi}\dfrac{d\xi}{\kappa^{-2}\left(  \left(  \kappa\xi-1\right)
^{s}+s\kappa\xi-1\right)  +\theta^{\left(  \kappa\right)  }}.
\]
Making the change $y=i\left(  \kappa\xi-1\right)  $ we get%
\[
I_{2}=-\dfrac{i\kappa e^{-i\frac{x}{\kappa}}}{\sqrt{2\pi}}%
{\displaystyle\int\limits_{0}^{\infty}}
e^{-\frac{x}{\kappa}y}\dfrac{dy}{\left(  -iy\right)  ^{s}-isy+s-1+\kappa
^{2}\theta^{\left(  \kappa\right)  }}.
\]
Integrating by parts we have%
\[
I_{2}=-\dfrac{i\kappa^{2}e^{-i\frac{x}{\kappa}}}{x\sqrt{2\pi}}\dfrac
{1}{s-1+\kappa^{2}\theta^{\left(  \kappa\right)  }}+\tilde{Y}_{1}+\tilde
{Y}_{2}+\tilde{Y}_{3},
\]
where%
\[
\tilde{Z}_{1}=\left(  -1\right)  ^{s+1}Z_{1}=\dfrac{s\left(  -i\right)
^{s+1}e^{-i\frac{x}{\kappa}}}{\left(  s-1\right)  ^{2}\sqrt{2\pi}}%
\Gamma\left(  s\right)  \left(  \frac{\kappa^{2+s}}{x^{s+1}}\right)
\]%
\[
\tilde{Z}_{2}=-\dfrac{s\kappa^{2}e^{-i\frac{x}{\kappa}}}{x\sqrt{2\pi}}%
{\displaystyle\int\limits_{0}^{\infty}}
e^{-\frac{x}{\kappa}y}\tfrac{1}{\left(  \left(  -iy\right)  ^{s}%
-isy+s-1+\kappa^{2}\theta^{\left(  \kappa\right)  }\right)  ^{2}}dy
\]
and%
\[
\tilde{Z}_{3}=\dfrac{\left(  -i\right)  ^{s+1}s\kappa^{2}e^{-i\frac{x}{\kappa
}}}{x\sqrt{2\pi}}%
{\displaystyle\int\limits_{0}^{\infty}}
e^{-\frac{x}{\kappa}y}y^{s-1}\left(  \tfrac{1}{\left(  \left(  -iy\right)
^{s}-isy+s-1+\kappa^{2}\theta^{\left(  \kappa\right)  }\right)  ^{2}}%
-\tfrac{1}{\left(  s-1\right)  ^{2}}\right)  dy.
\]
Changing $z=\frac{x}{\kappa}y$ in $\tilde{Z}_{3}$ we show that%
\[
\tilde{Z}_{3}=\frac{o\left(  \kappa^{2+s}\right)  }{x^{s+1}}+\kappa
^{2+s}o\left(  \frac{1}{x^{s+1}}\right)  .
\]
Then,%
\begin{equation}
\left.
\begin{array}
[c]{c}%
I_{2}=\dfrac{s\left(  -i\right)  ^{s+1}e^{-i\frac{x}{\kappa}}}{\left(
s-1\right)  ^{2}\sqrt{2\pi}}\Gamma\left(  s\right)  \left(  \frac{\kappa
^{2+s}}{x^{s+1}}\right)  -\dfrac{i\kappa^{2}}{x\sqrt{2\pi}}\dfrac
{e^{-i\frac{x}{\kappa}}}{s-1+\kappa^{2}\theta^{\left(  \kappa\right)  }%
}+\tilde{Z}_{2}\\
+\frac{\kappa^{2+s}}{x^{s+1}}o_{\kappa}\left(  1\right)  +\frac{\kappa^{2+s}%
}{x^{s+1}}o_{\left\vert x\right\vert }\left(  1\right)  .
\end{array}
\right.  \label{tw144}%
\end{equation}
Using (\ref{tw142}) and (\ref{tw144}) in (\ref{tw145}) and noting that
$Z_{2}+\tilde{Z}_{2}=O\left(  \frac{\kappa^{3+s}}{x^{s+2}}\right)  ,$ we
obtain (\ref{as45}) for $x>0.$ The case $x<0$ is considered similarly.
\end{proof}

We now get a bound for a solution of (\ref{tw148}).\ Namely, we prove the following.

\begin{lemma}
\label{L11}Let $R_{N}\in H^{s/2}$ be a solution to (\ref{tw148}) with
$m_{N}\left(  x\right)  \in L^{\infty}$ satisfying the decay estimate%
\begin{equation}
\left\vert m_{N}\left(  x\right)  \right\vert \leq C\left(  e^{-\sqrt
{\lambda\left(  s\right)  }\left\vert x\right\vert }+\frac{N^{\frac{s\left(
2+s\right)  }{2-s}}}{1+\left\vert x\right\vert ^{s+1}}\right)  , \label{tw157}%
\end{equation}
for all $0<N\leq N_{0}$ and some $N_{0}>0.$ Then, there is $N_{1}>0,$ such
that for all $0<N\leq N_{1},$ the estimate
\begin{equation}
\left\vert R_{N}\left(  x\right)  \right\vert \leq C\left(  e^{-\sqrt
{\lambda\left(  s\right)  }\left\vert x\right\vert }+\frac{N^{\frac{s\left(
2+s\right)  }{2-s}}}{1+\left\vert x\right\vert ^{s+1}}\right)  ,\text{ }%
x\in\mathbb{R}\text{,} \label{tw167}%
\end{equation}
is true.
\end{lemma}

\begin{proof}
Suppose that%
\[
\left\vert m_{N}\left(  x\right)  \right\vert \leq\frac{CA}{1+\left\vert
x\right\vert ^{s+1}},\text{ }C,A>0.
\]
Let us prove that
\begin{equation}
\left\vert R_{N}\left(  x\right)  \right\vert \leq\frac{CA}{1+\left\vert
x\right\vert ^{s+1}}. \label{tw150}%
\end{equation}
To show this inequality, we follow the proof of Lemma 3.1 of \cite{gerard}.
Since $R_{N}\in L^{2},$ there is $a_{0}>0$ such that
\[
\left(  \int_{\left\vert y\right\vert \geq a_{0}/2}R_{N}^{2}\left(  y\right)
dy\right)  ^{\frac{1}{2}}\leq\frac{1}{16}.
\]
For any $a>0$, we set%
\[
M\left(  a\right)  :=\mathop{\rm sup}_{\left\vert x\right\vert \geq
a}\left\vert R_{N}\left(  x\right)  \right\vert .
\]
Note that%
\[
\left.
\begin{array}
[c]{c}%
\left\vert \left(  m_{N}\ast\left(  \left\vert R_{N}\right\vert ^{2s}%
R_{N}\right)  \right)  \left(  x\right)  \right\vert \leq\left\vert
\int_{\left\vert y\right\vert \leq\frac{a}{2}}m_{N}\left(  x-y\right)  \left(
\left\vert R_{N}\right\vert ^{2s}R_{N}\right)  \left(  y\right)  dy\right\vert
\\
+\left\vert \int_{\left\vert y\right\vert \geq\frac{a}{2}}m_{N}\left(
x-y\right)  \left(  \left\vert R_{N}\right\vert ^{2s}R_{N}\right)  \left(
y\right)  dy\right\vert \\
\leq A\left(  \dfrac{C}{1+a^{s+1}}+\frac{1}{16}M\left(  \dfrac{a}{2}\right)
\right)  ,
\end{array}
\right.
\]
for all $\left\vert x\right\vert \geq a,$ and $a\geq a_{0}.$ Using equation
(\ref{tw148}), from the last relation we deduce%
\[
M\left(  a\right)  \leq A\left(  \dfrac{C}{1+a^{s+1}}+\frac{1}{16}M\left(
\dfrac{a}{2}\right)  \right)  ,
\]
for all $a\geq a_{0}.$ Putting $a=2^{n},$ $n\geq n_{0}$, in last relation we
obtain%
\[
M\left(  2^{n}\right)  \leq A\left(  \dfrac{C}{1+\left(  2^{n}\right)  ^{s+1}%
}+\frac{1}{16}M\left(  2^{n-1}\right)  \right)  .
\]
Iterating the above inequality we see that
\begin{equation}
\left.
\begin{array}
[c]{c}%
M\left(  2^{n}\right)  \leq A\left(  \dfrac{C}{1+\left(  2^{n}\right)  ^{s+1}%
}\sum_{j=0}^{n-n_{0}}2^{-j}+\frac{1}{\left(  16\right)  ^{n-n_{0}+1}}M\left(
2^{n_{0}-1}\right)  \right) \\
\leq\dfrac{CA}{1+\left(  2^{n}\right)  ^{s+1}}.
\end{array}
\right.  \label{tw151}%
\end{equation}
Since $R_{N}\in L^{\infty},$ to prove (\ref{tw150}) we can assume that
$\left\vert x\right\vert $ is big enough. For instance, $\left\vert
x\right\vert \geq2^{n_{0}}.$ Then, as $\left\vert x\right\vert \in\lbrack
2^{n},2^{n+1}],$ for some $n\geq n_{0},$ we deduce (\ref{tw150}) from
(\ref{tw151}).

Suppose first that $N^{\frac{s\left(  2+s\right)  }{2-s}}\geq e^{-\sqrt
{\lambda\left(  s\right)  }\left\vert x\right\vert }\left(  1+\left\vert
x\right\vert ^{s+1}\right)  .$ Then, from (\ref{tw157}) it follows that%
\[
\left\vert m_{N}\left(  x\right)  \right\vert \leq C\frac{N^{\frac{s\left(
2+s\right)  }{2-s}}}{1+\left\vert x\right\vert ^{s+1}}.
\]
Applying (\ref{tw150}) with $A=N^{\frac{s\left(  2+s\right)  }{2-s}}$ we get%
\begin{equation}
\left\vert R_{N}\left(  x\right)  \right\vert \leq\frac{CN^{\frac{s\left(
2+s\right)  }{2-s}}}{1+\left\vert x\right\vert ^{s+1}}. \label{tw158}%
\end{equation}
Suppose now that
\begin{equation}
N^{\frac{s\left(  2+s\right)  }{2-s}}\leq e^{-\sqrt{\lambda\left(  s\right)
}\left\vert x\right\vert }\left(  1+\left\vert x\right\vert ^{s+1}\right)  .
\label{tw163}%
\end{equation}
Let us prove that%
\begin{equation}
\left\vert R_{N}\left(  x\right)  \right\vert \leq Ce^{-\sqrt{\lambda\left(
s\right)  }\left\vert x\right\vert }. \label{tw152}%
\end{equation}
Since
\[
\left\vert m_{N}\left(  x\right)  \right\vert \leq\frac{C}{1+\left\vert
x\right\vert ^{s+1}},
\]
it follows from (\ref{tw150}) that
\begin{equation}
\left\vert R_{N}\left(  x\right)  \right\vert \leq\frac{C}{1+\left\vert
x\right\vert ^{s+1}}. \label{tw159}%
\end{equation}
To prove exponential decay we turn to the equation (\ref{tw115}), which we
write as
\begin{equation}
\left(  -\Delta\right)  R_{N}+\theta_{N}R_{N}=\left\vert R_{N}\right\vert
^{2s}R_{N}+\left(  \left(  -\Delta\right)  R_{N}-n_{N}\left(  D\right)
\right)  R_{N}. \label{tw154}%
\end{equation}
Putting $R_{N}\left(  x\right)  =P_{N}\left(  \sqrt{\theta_{N}}x\right)  $ in
(\ref{tw154}) we have%
\[
\left(  -\Delta\right)  P_{N}+P_{N}=\frac{1}{\theta_{N}}\left\vert
P_{N}\right\vert ^{2s}P_{N}+\left(  \left(  -\Delta\right)  R_{N}-n_{N}\left(
D\right)  \right)  R_{N}%
\]
Following the proof of Theorem 8.1.1 of \cite{Cazenave}, we introduce the
function $\omega_{\varepsilon,\delta}\left(  x\right)  =e^{\frac
{\delta\left\vert x\right\vert }{1+\varepsilon\left\vert x\right\vert }},$ for
$\varepsilon,\delta>0.$ This function is bounded, Lipschitz continuous, and
$\left\vert \nabla\omega_{\varepsilon}\right\vert \leq\omega_{\varepsilon
,\delta}$ a.e.. Taking the scalar product of (\ref{tw154}) with $\omega
_{\varepsilon,\delta}P_{N}\in H^{s/2},$ we get%
\begin{equation}
\left.
\begin{array}
[c]{c}%
\operatorname{Re}\int\nabla P_{N}\nabla\left(  \omega_{\varepsilon,\delta
}\overline{P_{N}}\right)  +\int\omega_{\varepsilon,\delta}\left\vert
P_{N}\right\vert ^{2}\\
=\frac{1}{\theta_{N}}\int\omega_{\varepsilon,\delta}\left\vert P_{N}%
\right\vert ^{2s+2}+\left\vert \int\left(  \left(  \left(  -\Delta\right)
-n_{N}\left(  D\right)  \right)  R_{N}\right)  \left(  \omega_{\varepsilon
,\delta}\overline{R_{N}}\right)  \right\vert .
\end{array}
\right.  \label{tw156}%
\end{equation}
Since $\nabla\left(  \omega_{\varepsilon,\delta}\overline{P_{N}}\right)
=\left(  \nabla\omega_{\varepsilon,\delta}\right)  \overline{P_{N}}%
+\omega_{\varepsilon,\delta}\nabla\overline{P_{N}},$%
\[
\left.
\begin{array}
[c]{c}%
\operatorname{Re}\int\nabla P_{N}\nabla\left(  \omega_{\varepsilon,\delta
}\overline{P_{N}}\right)  \geq\int\omega_{\varepsilon,\delta}\left\vert \nabla
P_{N}\right\vert ^{2}-\int\omega_{\varepsilon,\delta}\left\vert P_{N}%
\right\vert \left\vert \nabla P_{N}\right\vert \\
\geq\int\omega_{\varepsilon,\delta}\left\vert \nabla P_{N}\right\vert
^{2}-\frac{1}{2}\int\omega_{\varepsilon,\delta}\left\vert P_{N}\right\vert
^{2}-\frac{1}{2}\int\omega_{\varepsilon,\delta}\left\vert \nabla
P_{N}\right\vert ^{2},
\end{array}
\right.
\]
and then from (\ref{tw156}) it follows%
\begin{equation}
\left.
\begin{array}
[c]{c}%
\int\omega_{\varepsilon,\delta}\left\vert P_{N}\right\vert ^{2}\leq\frac
{2}{\theta_{N}}\int\omega_{\varepsilon,\delta}\left\vert P_{N}\right\vert
^{2s+2}\\
+2\int\left(  \left(  \left(  -\Delta\right)  -n_{N}\left(  D\right)  \right)
R_{N}\right)  \left(  \omega_{\varepsilon,\delta}\overline{R_{N}}\right)  .
\end{array}
\right.  \label{tw161}%
\end{equation}
Using (\ref{tw159}) we have%
\[
\left.
\begin{array}
[c]{c}%
\frac{2}{\theta_{N}}\int\omega_{\varepsilon,\delta}\left\vert P_{N}\right\vert
^{2s+2}\leq\frac{2}{\theta_{N}}\int_{\left\vert y\right\vert \leq R}%
\omega_{\varepsilon,\delta}\left\vert P_{N}\left(  y\right)  \right\vert
^{2s+2}dy\\
+\frac{2C}{\theta_{N}\left(  1+R^{s+1}\right)  ^{2s}}\int_{\left\vert
y\right\vert \geq R}\omega_{\varepsilon,\delta}\left\vert P_{N}\left(
y\right)  \right\vert ^{2}dy\\
\leq\frac{2}{\theta_{N}}\int_{\left\vert y\right\vert \leq R}e^{\delta
\left\vert y\right\vert }\left\vert P_{N}\left(  y\right)  \right\vert
^{2s+2}dy+\frac{1}{2}\int\omega_{\varepsilon,\delta}\left\vert P_{N}%
\right\vert ^{2},
\end{array}
\right.
\]
for some $R>0$ big enough. Using the last relation in (\ref{tw161}) we get%
\begin{equation}
\left.
\begin{array}
[c]{c}%
\int\omega_{\varepsilon,\delta}\left(  \sqrt{\theta_{N}}y\right)  \left\vert
R_{N}\left(  y\right)  \right\vert ^{2}dy\\
\leq\frac{4}{\left(  \theta_{N}\right)  ^{3/2}}\int_{\left\vert y\right\vert
\leq R}e^{\delta\left\vert y\right\vert }\left\vert P_{N}\left(  y\right)
\right\vert ^{2s+2}dy\\
+\frac{4}{\sqrt{\theta_{N}}}\left\vert \int\left(  \left(  \left(
-\Delta\right)  -n_{N}\left(  D\right)  \right)  R_{N}\right)  \left(
\omega_{\varepsilon,\delta}\overline{R_{N}}\right)  \right\vert .
\end{array}
\right.  \label{tw162}%
\end{equation}
We estimate the second term in the right-hand side of (\ref{tw162}) by%
\begin{equation}
\left.
\begin{array}
[c]{c}%
\left\vert \int\left(  \left(  \left(  -\Delta\right)  -n_{N}\left(  D\right)
\right)  R_{N}\right)  \left(  \omega_{\varepsilon,\delta}\overline{R_{N}%
}\right)  \right\vert \\
\leq\int\left\vert \left\langle y\right\rangle \tfrac{\omega_{\varepsilon
,\delta}}{\sqrt{\omega_{\varepsilon,\delta}\left(  \sqrt{\theta_{N}}y\right)
}}\left(  \left(  \left(  -\Delta\right)  -n_{N}\left(  D\right)  \right)
R_{N}\right)  \right\vert \left(  \left\langle y\right\rangle ^{-1}%
\sqrt{\omega_{\varepsilon,\delta}\left(  \sqrt{\theta_{N}}y\right)
}\left\vert R_{N}\right\vert \right)  dy.
\end{array}
\right.  \label{tw164}%
\end{equation}
Using (\ref{tw163}) we have $\left\langle y\right\rangle \tfrac{\omega
_{\varepsilon,\delta}}{\sqrt{\omega_{\varepsilon,\delta}\left(  \sqrt
{\theta_{N}}y\right)  }}\leq N^{-k\delta},$ for some $k>0.$ Then, if
$\delta>0$ is small enough, by using (\ref{tw61}), via Sobolev's theorem we
obtain%
\[
\left.
\begin{array}
[c]{c}%
\left\vert \left\langle y\right\rangle \tfrac{\omega_{\varepsilon,\delta}%
}{\sqrt{\omega_{\varepsilon,\delta}\left(  \sqrt{\theta_{N}}y\right)  }%
}\left(  \left(  \left(  -\Delta\right)  -n_{N}\left(  D\right)  \right)
R_{N}\right)  \right\vert \\
\leq CN^{-k\delta}\left\Vert \left(  \left(  \left(  -\Delta\right)
-n_{N}\left(  D\right)  \right)  R_{N}\right)  \right\Vert _{H^{1}}\leq
CN^{\delta_{1}}\left\Vert R_{N}\right\Vert _{H^{n}},
\end{array}
\right.
\]
for some $\delta_{1}>0$ and $n>1.$ Then, there is $N_{0}>0,$ such that%
\[
\mathop{\rm sup}_{y\in\mathbb{R}}\left\vert \left\langle y\right\rangle
\tfrac{\omega_{\varepsilon,\delta}}{\sqrt{\omega_{\varepsilon,\delta}\left(
\sqrt{\theta_{N}}y\right)  }}\left(  \left(  \left(  -\Delta\right)
-n_{N}\left(  D\right)  \right)  R_{N}\right)  \right\vert \leq\frac
{\sqrt{\theta_{N}}}{8}.
\]
Using this inequality in (\ref{tw164}) we get%
\[
\left\vert \int\left(  \left(  \left(  -\Delta\right)  -n_{N}\left(  D\right)
\right)  R_{N}\right)  \left(  \omega_{\varepsilon,\delta}\overline{R_{N}%
}\right)  \right\vert \leq\frac{\sqrt{\theta_{N}}}{4}+\frac{\sqrt{\theta_{N}}%
}{8}\int\omega_{\varepsilon,\delta}\left(  \sqrt{\theta_{N}}y\right)
\left\vert R_{N}\left(  y\right)  \right\vert ^{2}dy.
\]
Introducing the last inequality into (\ref{tw162}) we arrive to%
\[
\int\omega_{\varepsilon,\delta}\left(  \sqrt{\theta_{N}}y\right)  \left\vert
R_{N}\left(  y\right)  \right\vert ^{2}dy\leq\frac{8}{\left(  \theta
_{N}\right)  ^{3/2}}\int_{\left\vert y\right\vert \leq R}e^{\delta\left\vert
x\right\vert }\left\vert P_{N}\left(  y\right)  \right\vert ^{2s+2}dy+2.
\]
Taking the limit as $\varepsilon\rightarrow0,$ we obtain%
\begin{equation}
\int e^{\sqrt{\theta_{N}}\delta\left\vert x\right\vert }\left\vert
R_{N}\left(  y\right)  \right\vert ^{2}dy\leq C. \label{tw165}%
\end{equation}
Note that by (\ref{tw163}) $\left\vert m_{N}\left(  x\right)  \right\vert \leq
Ce^{-\frac{2\lambda\left(  s\right)  }{s\left(  s-1\right)  }\left\vert
x\right\vert }.$ Then, it follows from (\ref{tw148}) that
\begin{equation}
\left\vert R_{N}\left(  x\right)  \right\vert \leq C\int e^{-\sqrt
{\lambda\left(  s\right)  }\left\vert x-y\right\vert }\left(  \left\vert
R_{N}\right\vert ^{2s}R_{N}\right)  \left(  y\right)  dy. \label{tw166}%
\end{equation}
Since $2s+1>2,$ from (\ref{tw165}) and (\ref{tw166}) we deduce (\ref{tw152}).
Finally, summing up (\ref{tw158}) and (\ref{tw152}) we attain (\ref{tw167}).
\end{proof}

We have now all ingredients that we need to prove Theorem \ref{Th4}.

\begin{proof}
[Proof of Theorem \ref{Th4}]We consider again equation (\ref{tw148}). It
follows from Theorem \ref{Th2} that there exist $\tilde{x},\tilde{\gamma}%
\in\mathbb{R}$, $\tilde{\gamma}=\tilde{\gamma}\left(  N\right)  $ and
$\tilde{x}=\tilde{x}\left(  N\right)  ,$ such that
\[
\mathop{\rm lim}_{N\rightarrow0}\left\Vert e^{i\tilde{\gamma}}R_{N}\left(
\cdot+\tilde{x}\right)  -\mathcal{R}\right\Vert _{H^{1}}=0.
\]
Let $\tilde{R}_{N}\left(  x\right)  =e^{i\tilde{\gamma}}R_{N}\left(
x+\tilde{x}\right)  .$ Note that $\tilde{R}_{N}$ solves (\ref{tw148}) and
tends to $\mathcal{R}$, as $N\rightarrow0.$ We write%
\begin{equation}
\left.
\begin{array}
[c]{c}%
\tilde{R}_{N}\left(  x\right)  =\mathcal{C}_{1}\int_{\left\vert y\right\vert
\leq\frac{\left\vert x\right\vert }{2}}e^{-\sqrt{\lambda\left(  s\right)
}\left\vert x-y\right\vert }\left(  \left\vert \mathcal{R}\right\vert
^{2s}\mathcal{R}\right)  \left(  y\right)  dy\\
+\frac{\mathcal{C}_{2}N^{\frac{s\left(  2+s\right)  }{2-s}}}{\left\vert
x\right\vert ^{s+1}}\int\left\vert \mathcal{R}\right\vert ^{2s}\mathcal{R}%
+\sum_{j=1}^{5}\rho_{j},
\end{array}
\right.  \label{tw170}%
\end{equation}
where%
\[
\rho_{1}:=\frac{\mathcal{C}_{2}N^{\frac{s\left(  2+s\right)  }{2-s}}%
}{\left\vert x\right\vert ^{s+1}}\int_{\left\vert y\right\vert \geq
\frac{\left\vert x\right\vert }{2}}\left(  \left\vert \mathcal{R}\right\vert
^{2s}\mathcal{R}\right)  \left(  y\right)  dy
\]%
\[
\rho_{2}:=\mathcal{C}_{2}N^{\frac{s\left(  2+s\right)  }{2-s}}\int_{\left\vert
y\right\vert \leq\frac{\left\vert x\right\vert }{2}}\left(  \frac
{1}{\left\vert x-y\right\vert ^{s+1}}-\frac{1}{\left\vert x\right\vert ^{s+1}%
}\right)  \left(  \left\vert \mathcal{R}\right\vert ^{2s}\mathcal{R}\right)
\left(  y\right)  dy
\]%
\[
\rho_{3}:=\int_{\left\vert y\right\vert \leq\frac{\left\vert x\right\vert }%
{2}}\left(  m_{N}\left(  x-y\right)  -\mathcal{C}_{1}e^{-\sqrt{\lambda\left(
s\right)  }\left\vert x-y\right\vert }-\mathcal{C}_{2}\frac{N^{\frac{s\left(
2+s\right)  }{2-s}}}{\left\vert x-y\right\vert ^{s+1}}\right)  \left(
\left\vert \mathcal{R}\right\vert ^{2s}\mathcal{R}\right)  \left(  y\right)
dy,
\]%
\[
\rho_{4}:=\int_{\left\vert y\right\vert \leq\frac{\left\vert x\right\vert }%
{2}}m_{N}\left(  x-y\right)  \left(  \left(  \left\vert \tilde{R}%
_{N}\right\vert ^{2s}\tilde{R}_{N}\right)  \left(  y\right)  -\left(
\left\vert \mathcal{R}\right\vert ^{2s}\mathcal{R}\right)  \left(  y\right)
\right)  dy
\]
and%
\[
\rho_{5}:=\int_{\left\vert y\right\vert \geq\frac{\left\vert x\right\vert }%
{2}}m_{N}\left(  x-y\right)  \left(  \left\vert \tilde{R}_{N}\right\vert
^{2s}\tilde{R}_{N}\right)  \left(  y\right)  dy.
\]
We observe that $\mathcal{R}$ satisfies the estimate%
\begin{equation}
\left\vert \mathcal{R}\left(  x\right)  \right\vert \leq Ce^{-\sqrt
{\lambda\left(  s\right)  }\left\vert x\right\vert },\text{ \ }x\in
\mathbb{R}\text{.} \label{tw169}%
\end{equation}
Then, as $s>1,$ $\rho_{1}=N^{\frac{s\left(  2+s\right)  }{2-s}}e^{-\sqrt
{\lambda\left(  s\right)  }\left\vert x\right\vert }o_{\left\vert x\right\vert
}\left(  1\right)  ,$ where $o_{\left\vert x\right\vert }\left(  1\right)
\rightarrow0,$ as $\left\vert x\right\vert \rightarrow\infty.$ Taking into
account the relation%
\[
\left.
\begin{array}
[c]{c}%
\left\vert \int_{\left\vert y\right\vert \leq\frac{\left\vert x\right\vert
}{2}}\left(  \frac{1}{\left\vert x-y\right\vert ^{s+1}}-\frac{1}{\left\vert
x\right\vert ^{s+1}}\right)  \left(  \left\vert \mathcal{R}\right\vert
^{2s}\mathcal{R}\right)  \left(  y\right)  dy\right\vert \\
\leq\frac{C}{\left\vert x\right\vert ^{s+2}}\int_{\left\vert y\right\vert
\leq\frac{\left\vert x\right\vert }{2}}\left\vert y\right\vert \left\vert
\left\vert \mathcal{R}\right\vert ^{2s}\mathcal{R}\right\vert \left(
y\right)  dy,
\end{array}
\right.
\]
we show that $\rho_{2}=\frac{N^{\frac{s\left(  2+s\right)  }{2-s}}}{\left\vert
x\right\vert ^{s+1}}o_{\left\vert x\right\vert }\left(  1\right)  .$ From
Lemma \ref{L10} it follows that
\[
\rho_{3}=o_{N}\left(  1\right)  e^{-\sqrt{\lambda\left(  s\right)  }\left\vert
x\right\vert }+\frac{N^{\frac{s\left(  2+s\right)  }{2-s}}o_{N}\left(
1\right)  }{\left\vert x\right\vert ^{s+1}}+\frac{N^{\frac{s\left(
2+s\right)  }{2-s}}}{\left\vert x\right\vert ^{s+1}}o\left(  1\right)  ,
\]
with $o_{N}\left(  1\right)  \rightarrow0,$ as $N\rightarrow0.$ Since
$m_{N}\in H^{\frac{1}{2}+\varepsilon},$ for $\varepsilon>0,$ in particular
$m_{N}\in L^{\infty}.$ Then, using Lemma \ref{L10} we show that $m_{N}$
satisfies estimate (\ref{tw157}). Hence, since $\tilde{R}_{N}\rightarrow
\mathcal{R}$, as $N\rightarrow0,$ we estimate
\[
\rho_{4}=o_{N}\left(  1\right)  \left(  e^{-\sqrt{\lambda\left(  s\right)
}\left\vert x\right\vert }+\frac{1}{\left\vert x\right\vert ^{s+1}}\right)  .
\]
Finally, as $s>1,$ by using Lemma \ref{L11} we deduce that
\[
\rho_{5}=e^{-\sqrt{\lambda\left(  s\right)  }\left\vert x\right\vert
}o_{\left\vert x\right\vert }\left(  1\right)  +\frac{N^{\frac{s\left(
2+s\right)  }{2-s}}}{1+\left\vert x\right\vert ^{s+1}}o_{\left\vert
x\right\vert }\left(  1\right)  o_{N}\left(  1\right)  .
\]
\ Therefore, summing up the above estimates we prove
\[
\sum_{j=1}^{5}\rho_{j}=\left(  e^{-\sqrt{\lambda\left(  s\right)  }\left\vert
x\right\vert }+\frac{N^{\frac{s\left(  2+s\right)  }{2-s}}}{\left\vert
x\right\vert ^{s+1}}\right)  \left(  o_{\left\vert x\right\vert }\left(
1\right)  +o_{N}\left(  1\right)  \right)  .
\]
As $\mathcal{R}$ is radially symmetric we have%
\begin{align*}
&  \mathcal{C}_{1}\int_{\left\vert y\right\vert \leq\frac{\left\vert
x\right\vert }{2}}e^{-\sqrt{\lambda\left(  s\right)  }\left\vert
x-y\right\vert }\left(  \left\vert \mathcal{R}\right\vert ^{2s}\mathcal{R}%
\right)  \left(  y\right)  dy\\
&  =\mathcal{C}_{1}e^{-\sqrt{\lambda\left(  s\right)  }\left\vert x\right\vert
}\int_{\left\vert y\right\vert \leq\frac{\left\vert x\right\vert }{2}}%
e^{\sqrt{\lambda\left(  s\right)  }y}\left(  \left\vert \mathcal{R}\right\vert
^{2s}\mathcal{R}\right)  \left(  y\right)  dy.
\end{align*}
Using this equality in (\ref{tw170}) we obtain%
\begin{align*}
\tilde{R}_{N}\left(  x\right)   &  =\mathcal{C}_{1}e^{-\sqrt{\lambda\left(
s\right)  }\left\vert x\right\vert }\int e^{\sqrt{\lambda\left(  s\right)  }%
y}\left(  \left\vert \mathcal{R}\right\vert ^{2s}\mathcal{R}\right)  \left(
y\right)  dy\\
&  +\frac{\mathcal{C}_{2}N^{\frac{s\left(  2+s\right)  }{2-s}}}{\left\vert
x\right\vert ^{s+1}}\int\left\vert \mathcal{R}\right\vert ^{2s}\mathcal{R}%
+\sum_{j=1}^{6}\rho_{j},
\end{align*}
with%
\[
\rho_{6}:=-\mathcal{C}_{1}e^{-\sqrt{\lambda\left(  s\right)  }\left\vert
x\right\vert }\int_{\left\vert y\right\vert \geq\frac{\left\vert x\right\vert
}{2}}e^{\sqrt{\lambda\left(  s\right)  }y}\left(  \left\vert \mathcal{R}%
\right\vert ^{2s}\mathcal{R}\right)  \left(  y\right)  dy
\]
By using (\ref{tw169}) we estimate $\rho_{6}=e^{-\sqrt{\lambda\left(
s\right)  }\left\vert x\right\vert }o_{\left\vert x\right\vert }\left(
1\right)  .$ Theorem \ref{Th4} is proved.
\end{proof}

\end{document}